\documentclass[11pt]{amsart}
\usepackage{graphicx}
\newtheorem{thm}{Theorem}[section]
\newtheorem{cor}[thm]{Corollary}

\newtheorem{lem}[thm]{Lemma}
\newtheorem{prop}[thm]{Proposition}
\theoremstyle{definition}
\newtheorem{defn}[thm]{Definition}
\theoremstyle{remark}
\newtheorem{rem}[thm]{Remark}
\theoremstyle{conclusion}

\numberwithin{equation}{section}

\begin{document}
\title[ Scattering theory for NLS in weighted Sobolev space]
{ Some Results on the Scattering Theory for Nonlinear Schr\"{o}dinger Equations in Weighted $L^{2}$ Space }

\author{Wei Dai}

\address{Institute of Applied Mathematics, AMSS, Chinese Academy of Sciences, Beijing 100190, P.R.China}
\email{daiwei@amss.ac.cn}

\address{Department of Mathematics, University of California, Berkeley,
  CA 94720, USA}
\email{weidai@math.berkeley.edu}

\begin{abstract}

We investigate the scattering theory for the nonlinear Schr\"{o} \\ dinger
equation $i \partial_{t}u+ \Delta u+\lambda|u|^\alpha
u=0$ in $\Sigma=H^{1}(\mathbb{R}^{d})\cap L^{2}(|x|^{2};dx)$. We show that scattering states $u^{\pm}$ exist in $\Sigma$ when $\alpha_{d}<\alpha<\frac{4}{d-2}$, $d\geq3$, $\lambda\in \mathbb{R}$ with certain smallness assumption on the initial data $u_{0}$, and when $\alpha(d)\leq \alpha< \frac{4}{d-2}$($\alpha\in [\alpha(d), \infty)$, if $d=1,2$), $\lambda>0$ under suitable conditions on $u_{0}$, where $\alpha_{d}$, $\alpha(d)$ are the positive root of the polynomial $dx^{2}+dx-4$ and $dx^{2}+(d-2)x-4$ respectively. Specially, when $\lambda>0$, we obtain the existence of $u^{\pm}$ in $\Sigma$ for $u_{0}$ below a mass-energy threshold $M[u_{0}]^{\sigma}E[u_{0}]<\lambda^{-2\tau}M[Q]^{\sigma}E[Q]$ and satisfying an mass-gradient bound $\|u_{0}\|_{L^{2}}^{\sigma}\|\nabla u_{0}\|_{L^{2}}<\lambda^{-\tau}\|Q\|_{L^{2}}^{\sigma}\|\nabla Q\|_{L^{2}}$ with $\frac{4}{d}<\alpha<\frac{4}{d-2}$($\alpha\in (\frac{4}{d}, \infty)$, if $d=1,2$), and also for oscillating data at critical power $\alpha=\alpha(d)$, where $\sigma=\frac{4-(d-2)\alpha}{\alpha d-4}$, $\tau=\frac{2}{\alpha d-4}$ and $Q$ is the ground state. We also study the convergence of $u(t)$ to the free solution $e^{it\Delta}u^{\pm}$ in $\Sigma$, where $u^{\pm}$ is the scattering state at $\pm\infty$ respectively.

\end{abstract}
\maketitle {\small {\bf Keywords: Nonlinear Schr\"{o}dinger equation; Scattering theory;
Oscillating data; Weighted spaces; Lorentz space} \\

}

\section{INTRODUCTION}

In this paper we study the scattering theory for the nonlinear Schr\"{o}dinger equation
\begin{equation}\label{eq1}
    \left\{
  \begin{array}{ll}
    i \partial_{t}u+ \Delta u+\lambda|u|^{\alpha}u=0, \,\,\, t\in \mathbb{R}, \, x\in \mathbb{R}^{d}\\
    u(0,x)=u_{0}(x)\in \Sigma, \,\,\, x\in \mathbb{R}^{d}
  \end{array}
\right.
\end{equation}
in weighted space $\Sigma=H^{1}(\mathbb{R}^{d})\cap L^{2}(|x|^{2};dx)$, where $d$ denotes the spatial dimension, $\lambda\in \mathbb{R}\setminus \{0\}$ and $0<\alpha<\frac{4}{d-2}$ ($0<\alpha<\infty$ if $d=1,2$).

As is well-known, if $\lambda<0$, or $\lambda>0$ and $\alpha<4/d$, the unique solution $u(t)$ to the Cauchy problem (\ref{eq1}) is global in time and bounded in $H^{1}(\mathbb{R}^{d})$, and $u\in C(\mathbb{R},\Sigma)$ (see e.g. \cite{C2}). If $\lambda>0$, $\frac{4}{d}\leq \alpha<\frac{4}{d-2}$ ($4/d\leq \alpha<\infty$, if $d=1,2$), the local well-posedness in $\Sigma$ has been established by using Kato's fixed point method to the equivalent integral equation(Duhamel's formula)
\begin{equation}\label{eq2}
    u(t)=e^{it\Delta}u_{0}+i\lambda\int^{t}_{0}e^{i(t-\tau)\Delta}|u(\tau)|^{\alpha}u(\tau) d\tau,
\end{equation}
in an appropriate space (see \cite{C2,G1,K1}), where $(e^{it\Delta})_{t\in \mathbb{R}}$ is the one parameter Schr\"{o}dinger group. More precisely, given $u_{0}\in \Sigma$, there exists $T>0$ and a unique solution $u\in C([-S,T],\Sigma)$ of (\ref{eq1}), which can be extended to a maximal existence interval $(-T_{min},T_{max})$. This solution either exist globally or blow up in finite time, the global versus blow-up dichotomy is associated inseparably with the mass-energy threshold condition of the initial data $u_{0}$(see \cite{D1,F1,H1,K2}). Moreover, for arbitrary $u_{0}\in\Sigma$, the corresponding solution $u(t)$ satisfies the mass and energy conservation laws:
\begin{equation}\label{conservation1}
    M[u(t)]=\int_{\mathbb{R}^{d}}|u(t,x)|^{2}dx=M[u_{0}],
\end{equation}
\begin{equation}\label{conservation2}
    E[u(t)]=\frac{1}{2}\int_{\mathbb{R}^{d}}|\nabla u(t,x)|^{2}dx-\frac{\lambda}{\alpha+2}\int_{\mathbb{R}^{d}}|u(t,x)|^{\alpha+2}dx=E[u_{0}],
\end{equation}
and the pseudo-conformal conservation law
\begin{equation}\label{conformal-conservation3}
    \frac{d}{dt}(\|(x+2it\nabla)u(t)\|_{L^{2}}^{2}-\frac{8\lambda t^{2}}{\alpha+2}\|u(t)\|_{L^{\alpha+2}}^{\alpha+2})=4\lambda \frac{\alpha d-4}{\alpha+2}t\|u(t)\|_{L^{\alpha+2}}^{\alpha+2}.
\end{equation}
Thus we will denote the mass and energy by $M[u]$ and $E[u]$ respectively, with no reference to the time $t$.

If the solution $u(t)$ is global in time, we will care about its asymptotic behavior as $t\rightarrow\pm\infty$. To state our results on this topic, we will introduce some basic notions of scattering theory (see \cite{C2}) below.

Let $u_{0}\in\Sigma$ be such that the corresponding solution u of (\ref{eq1}) is defined for all $t\geq 0$, i.e., $T_{max}=\infty$. If the limit
\begin{equation}\label{u+}
    u^{+}=\lim_{t\rightarrow +\infty}e^{-it\Delta}u(t)
\end{equation}
exists in $\Sigma$, we say that $u^{+}$ is the scattering state of $u_{0}$ at $+\infty$. Also, if $u_{0}\in\Sigma$ is such that the solution of (\ref{eq1}) is defined for all $t\leq 0$, i.e., $T_{min}=\infty$, and if the limit
\begin{equation}\label{u-}
    u^{-}=\lim_{t\rightarrow -\infty}e^{-it\Delta}u(t)
\end{equation}
exists in $\Sigma$, we say that $u^{-}$ is the scattering state of $u_{0}$ at $-\infty$.

We observe that saying that $u_{0}$ has a scattering state at $\pm\infty$ is a way of saying that $u(t)$ behaves as $t\rightarrow\pm\infty$ like the solution $e^{it\Delta}u^{\pm}$ of the linear Schr\"{o}dinger equation. We set
\begin{equation}\label{R+}
    \mathcal{R}_{+}=\{u_{0}\in\Sigma: T_{max}=\infty \,\,\,and \,\,\,the \,\,\,limit \,\,\,(\ref{u+}) \,\,\,exists\}\,
\end{equation}
and
\begin{equation}\label{R-}
    \mathcal{R}_{-}=\{u_{0}\in\Sigma: T_{min}=\infty \,\,\,and \,\,\,the \,\,\,limit \,\,\,(\ref{u-}) \,\,\,exists\},
\end{equation}
which denote the set of initial values $u_{0}$ that have a scattering state at $\pm\infty$.

\begin{rem}\label{rem1}
We can see that changing $t$ to $-t$ in the equation (\ref{eq1}) corresponds to changing $u$ to $\overline{u}$, which means changing $u_{0}$ to $\overline{u}_{0}$. So we have
\[\mathcal{R}_{-}=\overline{\mathcal{R}_{+}}=\{u_{0}\in \Sigma: \overline{u}_{0}\in \mathcal{R}_{+}\}.\]
\end{rem}

The scattering theory for (\ref{eq1}) in weighted space $\Sigma$ has been quite extensively studied (see \cite{C1,G1,G2,N3,S1,T2,T3}). It is well-known that if $\alpha\leq \frac{2}{d}$, then no scattering theory can be developed for equation (\ref{eq1})(see \cite{T2}). In the defocusing case $\lambda<0$, low energy scattering theory holds in $\Sigma$ provided $4/(d+2)<\alpha<4/(d-2)$($2<\alpha<\infty$, if $d=1$). Moreover, if $\alpha\geq \alpha(d)=\frac{2-d+\sqrt{d^{2}+12d+4}}{2d}$, then scattering theory holds in whole $\Sigma$ space(see \cite{G1,N3,T3}), and we notice that the asymptotic completeness for $\frac{2}{d}<\alpha<\alpha(d)$ was established recently in \cite{S2}. For the focusing case $\lambda>0$, to our best knowledge, there is only a little positive answers, there is no low energy scattering if $\alpha<4/(d+2)$, but when $\alpha>4/(d+2)$, a low energy scattering theory holds in $\Sigma$(see \cite{C1}). If $\alpha\geq \frac{4}{d}$, some solutions will blow up in finite time.

In this paper we are mainly concerned with the scattering theory in $\Sigma$ for focusing NLS, but we also study the convergence of a global solution $u(t)$ of (\ref{eq1}) to the free solutions generated by its scattering states $u^{\pm}$ in $\Sigma$, our main results is Theorem \ref{small data}, \ref{XW}, \ref{rapidly decaying}, \ref{Rapid decay data}, \ref{characterization of R}, \ref{mass supercritical}, \ref{convergence} and Corollary \ref{Oscillating data}.

The content of our paper can be mainly divided into four parts.

1. Cazenave and Weissler have proved in \cite{C1} that if $\alpha>\frac{4}{d+2}$ and $\|u_{0}\|_{\Sigma}$ is small, then scattering states $u^{\pm}$ exist in $\Sigma$ at $\pm\infty$. They also have shown in \cite{C1} that if $\lambda>0$, $\alpha<\frac{4}{d+2}$, then the scattering theory for small data in $\Sigma$ fails, there are initial values $u_{0}\in\Sigma$ with arbitrary small norm $\|u_{0}\|_{\Sigma}$ that do not have a scattering state, even in the sense of $L^{2}(\mathbb{R}^{d})$. However, by applying the pseudo-conformal transformation and studying the resulting nonautonomous NLS, for $\lambda\in\mathbb{R}\setminus\{0\}$, $d\geq 3$, $u_{0}\in \Sigma$, let $v_{0}=e^{-\frac{i|x|^{2}}{4}}u_{0}$, we show in Theorem \ref{small data} that if $\alpha>\alpha_{d}=\frac{-d+\sqrt{d^{2}+16d}}{2d}$ and $u_{0}$ is such that $\|v_{0}\|_{H^{2}}$ is small enough, then scattering states $u^{\pm}$ exist in $\Sigma$ at $\pm\infty$. Note that $d\geq 3$, $\max\{\frac{2}{d},\frac{4}{d+4}\}<\alpha_{d}<\frac{4}{d+2}$, thus Theorem \ref{small data} extends the scattering theory for small initial values to the range $\alpha_{d}<\alpha\leq \frac{4}{d+2}$.

2. Next, we consider the scattering theory for the focusing NLS with $\alpha\geq \alpha(d)=\frac{2-d+\sqrt{d^{2}+12d+4}}{2d}$, it's obvious that this problem is closely associated with the decay property of the solution $u(t)$. Cazenave and Weissler have proposed in \cite{C1} a notion of ``positively rapidly decaying solutions" (i.e., a positively global solution $u$ of (\ref{eq1}) which satisfies $\|u\|_{L^{a}((0,\infty),L^{\alpha+2})}<\infty$, where $a=\frac{2\alpha(\alpha+2)}{4-\alpha(d-2)}$, $0<\alpha<\frac{4}{d-2}$, and $0<\alpha<\infty$, if $d=1,2$), and characterized the sets $\mathcal{R}_{\pm}$ in the case $\lambda>0$, $\alpha>\alpha(d)$ in terms of rapidly decaying solutions(refer to \cite{C1}, Theorem 4.12). But we can see that if $\alpha\leq \alpha(d)$, the rapidly decaying solutions should decay faster than the optimal rate $t^{-\frac{\alpha d}{2(\alpha+2)}}$ as $t\rightarrow+\infty$, that's impossible unless $u\equiv 0$(see \cite{C1}, Proposition 3.15). Thus we give a refined definition of Rapidly Decaying Solution below:
\begin{defn}\label{RDS}
Suppose $\alpha(d)\leq \alpha<\frac{4}{d-2}$ ($\alpha(d)\leq \alpha<\infty$, if $d=1,2$). A positively global solution $u$ of (\ref{eq1}) is rapidly decaying if
\begin{equation}\label{D+1}
\|u\|_{L^{a,\infty}((0,\infty),L^{\alpha+2})}<\infty, \,\,\,\,\, for \,\,\, \alpha=\alpha(d);
\end{equation}
\begin{equation}\label{D+2}
\|u\|_{L^{a}((0,\infty),L^{\alpha+2})}<\infty, \,\,\,\,\, for \,\,\, \alpha>\alpha(d);
\end{equation}
where $a=\frac{2\alpha(\alpha+2)}{4-\alpha(d-2)}$, and $\|\cdot\|_{L^{a,\infty}}$ denotes the weak $L^{a}$ norm.
Correspondingly, we say a negatively global solution $u$ of (\ref{eq1}) is rapidly decaying if it satisfies (\ref{D+1}) or (\ref{D+2}) with the time interval changed into $(-\infty,0)$ respectively.
\end{defn}
In Theorem \ref{rapidly decaying} we show that if $\alpha=\alpha(d)$, $d\neq2$, the positively(resp. negatively) rapidly decaying solutions have scattering states at $+\infty$(resp. $-\infty$), and we characterized the sets $\mathcal{R}_{\pm}$ for $\lambda>0$ in terms of rapidly decaying solutions in Theorem \ref{characterization of R}. Our method is mainly based on the lower bound estimate of the $L^{\alpha+2}$ norm for singular solutions to the following nonautonomous equation defined on time interval $[0,1)$:
\begin{equation}\label{nonautonomous1}
    i\partial_{s}v+\Delta v+\lambda(1-s)^{\frac{\alpha d-4}{2}}|v|^{\alpha}v=0.
\end{equation}

We will prove in Corollary \ref{Oscillating data} that if $\alpha=\alpha(d)$, then for arbitrary $\varphi\in H^{2}\cap \mathcal{F}(H^{2})\subset\Sigma$, given $b\in \mathbb{R}$, and let $\widetilde{u}_{b}$ be the corresponding global solution of (\ref{eq1}) with the initial value $\widetilde{u}_{b,0}=e^{i\Delta}(e^{i\frac{b|x|^{2}}{4}}\varphi)\in\Sigma$, there exists $0<b_{0}<\infty$ such that if $b\geq b_{0}$(resp. $b\leq -b_{0}$), then the global solution $\widetilde{u}_{b}$ is positively(resp. negatively) rapidly decaying, therefore scattering states $u_{b}^{+}$ exist at $+\infty$ and $\widetilde{u}_{b,0}\in \mathcal{R}_{+}$(resp. $u_{b}^{-}$ exist at $-\infty$ and $\widetilde{u}_{b,0}\in \mathcal{R}_{-}$). Moreover, a byproduct is the sets $\mathcal{R}_{\pm}$ are unbounded subsets of $L^{2}(\mathbb{R}^{d})$(see Theorem \ref{characterization of R}).

3. We also investigate the scattering theory in $\Sigma$ under certain suitable assumptions on initial data $u_{0}$(see Theorem \ref{XW}, Theorem \ref{Rapid decay data} and Theorem \ref{mass supercritical}). Specially, note that for $\frac{4}{d}<\alpha\leq \frac{4}{d-2}$, there are a lot of literature devoted to the scattering theory for focusing NLS in $H^{1}(\mathbb{R}^{d})$ for initial data $u_{0}\in H^{1}(\mathbb{R}^{d})$ below a mass-energy threshold and satisfying an mass-gradient bound (see Kenig and Merle \cite{K2}, Killip and Visan \cite{K4} for the energy-critical case $\alpha=\frac{4}{d-2}$, \cite{D1} and \cite{H1} for the 3D cubic case, and \cite{F1} for the general energy-subcritical case). We will show in Theorem \ref{mass supercritical} that if $\lambda>0$, $\frac{4}{d}<\alpha<\frac{4}{d-2}$($\alpha\in (\frac{4}{d}, \infty)$, if $d=1,2$), assume $u_{0}\in\Sigma$ below a mass-energy threshold $M[u_{0}]^{\sigma}E[u_{0}]<\lambda^{-2\tau}M[Q]^{\sigma}E[Q]$ and satisfying an mass-gradient bound $\|u_{0}\|_{L^{2}}^{\sigma}\|\nabla u_{0}\|_{L^{2}}<\lambda^{-\tau}\|Q\|_{L^{2}}^{\sigma}\|\nabla Q\|_{L^{2}}$, then scattering states $u^{\pm}$ exist in $\Sigma$ at $\pm\infty$, where $\sigma=\frac{4-(d-2)\alpha}{\alpha d-4}$, $\tau=\frac{2}{\alpha d-4}$ and $Q$ is the ground state solution to $-\Delta Q+Q=|Q|^{\alpha}Q$.

4. Finally we study the asymptotic behavior of $\|u(t)-e^{it\Delta}u^{\pm}\|_{\Sigma}$ under the assumption $u^{\pm}$ exist at $\pm\infty$. In general, since $e^{it\Delta}$ is not an isometry of $\Sigma$, it is not known whether we can deduce $\|u(t)-e^{it\Delta}u^{\pm}\|_{\Sigma}\rightarrow 0$ from the scattering asymptotic property $\|e^{-it\Delta}u(t)-u^{\pm}\|_{\Sigma}\rightarrow 0$. A positive answer has been given by B\'{e}gout \cite{B1} for $d\leq 2$, $\alpha>\frac{4}{d}$, and $3\leq d\leq 5$, $\alpha>\frac{8}{d+2}$. Our paper extends this result under certain suitable conditions on $u_{0}$. Theorem \ref{convergence} will show that if we assume $u_{0}\in \Sigma\cap W^{1,\rho'}$, then we can obtain the convergence $\|u(t)-e^{it\Delta}u^{\pm}\|_{\Sigma}\rightarrow 0$ for $3\leq d\leq 9$, $\alpha>\frac{16}{3d+2}$, where $\rho=\frac{4d}{2d-\alpha(d-2)}$.

The rest of the paper is organized as follows. In Section 2, we will give some preliminaries and notations. In Section 3 we will prove Theorem \ref{small data} for the small initial data scattering and in Section 4 we prove a lower bound estimate for the $L^{\alpha+2}$ norm of singular solutions to the nonautonomous equation (\ref{nonautonomous1}). Section 5 is devoted to some $\Sigma$ scattering results for (\ref{eq1}) in the focusing case with $\alpha\geq\alpha(d)$ and Section 6 to the study on the asymptotic convergence of the scattering solution to a free solution in $\Sigma$. \\

\section{Notations and preliminaries}

\subsection{Some notations} Throughout this paper, we use the following notation. $\bar{z}$ is the conjugate of the complex number $z$, $\Re z$ and $\Im z$ are respectively the real the imaginary part of the complex number $z$. All function spaces involved are spaces of complex valued functions. We denote by $p'$ the conjugate of the exponent $p\in[1,\infty]$ defined by $\frac{1}{p}+\frac{1}{p'}=1$, and $L^{p}=L^{p}(\mathbb{R}^{d})=L^{p}(\mathbb{R}^{d};\mathbb{C})$ with norm $\|\cdot\|_{L^{p}}$; $H^{1}=H^{1}(\mathbb{R}^{d})=H^{1}(\mathbb{R}^{d};\mathbb{C})$ with norm $\|\cdot\|_{H^{1}}$; and for all $(f,g)\in L^{2}\times L^{2}$, the scalar product $(f,g)_{L^{2}}=\Re \int_{\mathbb{R}^{d}}f(x)\overline{g(x)}dx$. Let $L^{q}_{t}(\mathbb{R},L^{r}_{x}(\mathbb{R}^{d}))$  denote the mixed Banach space with norm defined by
\[\|u\|_{L^{q}(\mathbb{R},L^{r})}
=(\int_{\mathbb{R}}(\int_{\mathbb{R}^{d}}|u(t,x)|^{r}dx)^{q/r}dt)^{1/q},\]
with the usual modifications when $q$ or $r$ is infinity, or when
the domain $\mathbb{R}\times\mathbb{R}^{d}$ is replaced by
a smaller region of spacetime such as $I\times\mathbb{R}^{d}$.
In what follows positive constants will be denoted by $C$ and will change from line to line. If necessary, by $C(\star,\cdots,\star)$ we denote positive constants depending only on the quantities appearing in parentheses continuously.

We denote by $(e^{it\Delta})_{t\in\mathbb{R}}$ the Schr\"{o}dinger group, which is isometric on $H^{s}$ and $\dot{H}^{s}$ for every $s\geq0$, and satisfies the Dispersive estimate and Strichartz's estimates(for more details, see Keel and Tao \cite{K3}). We will use freely the well-known properties of of the Schr\"{o}dinger group $(e^{it\Delta})_{t\in\mathbb{R}}$ (see e.g. Chapter 2 of \cite{C2} for an account of these properties). In convenience, we will introduce the definition of ``admissible pair" below, which plays an important role in space-time estimates.
\begin{defn}
We say that a pair $(q,r)$ is admissible if
\begin{equation*}\label{admissible pair}
    \frac{2}{q}=\delta(r)=d(\frac{1}{2}-\frac{1}{r})
\end{equation*}
and $2\leq r\leq\frac{2d}{d-2}$ ($2\leq r\leq\infty$ if $d=1$, $2\leq r<\infty$ if $d=2$). Note that if $(q,r)$ is an admissible pair, then $2\leq q\leq\infty$, the pair $(\infty,2)$ is always admissible, and the pair $(2,\frac{2d}{d-2})$ is admissible if $d\geq3$.
\end{defn}

\subsection{Generalized H\"{o}lder's and Young's inequality in Lorentz spaces} Our paper involves estimates in the general Lorentz spaces $L^{p,q}(0<p<\infty, 0<q\leq \infty)$ equipped with the norm
\[\|f\|_{L^{p,q}(X,\mu)}=p^{1/q}\|\lambda\mu(\{|f|\geq\lambda\})^{1/p}\|_{L^{q}(\mathbb{R}^{+},\frac{d\lambda}{\lambda})}\]
(refer to \cite{T1} for a review). A special case is the weak $L^{p}$ space $L^{p,\infty}(0<p<\infty)$ with norm defined by $\|f\|_{L^{p,\infty}(X,\mu)}=\sup_{\lambda>0}\lambda\mu(\{|f|\geq\lambda\})^{1/p}$.
An useful formula is for any $0<p,r<\infty$ and $0<q\leq\infty$,
\begin{equation}\label{Lorentz1}
    \|f^{r}\|_{L^{p,q}(X,\mu)}\sim_{p,q,r}\|f\|_{L^{pr,qr}(X,\mu)}^{r}.
\end{equation}
If $f:\mathbb{R}^{+}\rightarrow\mathbb{R}^{+}$ is a monotone non-increasing function, we have
\begin{equation}\label{Lorentz0}
    \|f\|_{L^{p,q}(X,\mu)}=\|f(t)t^{1/p}\|_{L^{q}(\mathbb{R}^{+},\frac{dt}{t})}.
\end{equation}
Below we give the refined H\"{o}lder's and Young's inequality for Lorentz spaces $L^{p,q}$, due to O'Neil(see \cite{T1}and \cite{N2} for the proof).
\begin{lem}\label{Lorentz2}
If $0<p_{1},p_{2},p<\infty$ and $0<q_{1},q_{2},q\leq\infty$ obey $\frac{1}{p}=\frac{1}{p_{1}}+\frac{1}{p_{2}}$ and $\frac{1}{q}=\frac{1}{q_{1}}+\frac{1}{q_{2}}$, then
\begin{equation}\label{H in Lorentz}
    \|fg\|_{L^{p,q}}\leq C\|f\|_{L^{p_{1},q_{1}}}\|g\|_{L^{p_{2},q_{2}}}.
\end{equation}
If $1<p_{1},p_{2},p<\infty$ and $1\leq q_{1},q_{2},q\leq\infty$ obey $1+\frac{1}{p}=\frac{1}{p_{1}}+\frac{1}{p_{2}}$ and $\frac{1}{q}=\frac{1}{q_{1}}+\frac{1}{q_{2}}$, then
\begin{equation}\label{Y in Lorentz}
    \|f\ast g\|_{L^{p,q}}\leq C\|f\|_{L^{p_{1},q_{1}}}\|g\|_{L^{p_{2},q_{2}}}.
\end{equation}
\end{lem}
If we use (\ref{Y in Lorentz}) with $q=q_{2}=2$ and $q_{1}=\infty$ instead of the Hardy-Littlewood-Sobolev inequality, we obtain the following Strichartz's estimate in Lorentz spaces.
\begin{lem}\label{Lorentz3} we have the following properties:
\[\|e^{i\Delta\cdot}\varphi\|_{L^{q,2}(\mathbb{R},L^{r})}\leq C\|\varphi\|_{L^{2}} \,\,\, for \,\,\, every \,\,\, \varphi\in L^{2}(\mathbb{R}^{d});\]
\[\|\int_{0}^{t}e^{i(t-\tau)\Delta}f(\tau)d\tau\|_{L^{q,2}(I,L^{r})\cap L^{\infty}(I,L^{2})}\leq C\|f\|_{L^{q',2}(I,L^{r'})},\]
for every $f\in L^{q',2}(I,L^{r'}(\mathbb{R}^{d}))$ and some constant $C$ independent of $I$, where $(q,r)$ is a admissible pair, $2<q<\infty$, $I$ is an interval of $\mathbb{R}$ such that $0\in \bar{I}$.
\end{lem}

\subsection{Properties of the operator $P_{t}=x+2it\nabla$}
Let $P_{t}$ be the partial differential operator on $\mathbb{R}^{d+1}$ defined by $P_{t}u(t,x)=(x+2it\nabla)u(t,x)$. Operator $P_{t}$ has the following important commutative properties:
\begin{equation}\label{P0}
    [P_{t},i\partial_{t}+\Delta]=0,
\end{equation}
\begin{equation}\label{P3}
    P_{t}e^{it\Delta}=e^{it\Delta}x, \,\,\,\,\,\,\,\, e^{-it\Delta}P_{t}=xe^{-it\Delta},
\end{equation}
Where $[\cdot,\cdot]$ is the commutator bracket. An easy calculation shows that if $t\neq 0$, then
\begin{equation}\label{P1}
    P_{t}u=(x+2it\nabla)u=2ite^{i\frac{|x|^{2}}{4t}}\nabla(e^{-i\frac{|x|^{2}}{4t}}u),
\end{equation}
and so \[\|(x+2it\nabla)u\|_{L^{2}}^{2}=4t^{2}\|\nabla(e^{-i\frac{|x|^{2}}{4t}}u)\|_{L^{2}}^{2}.\]
Let $v(t,x)=e^{-i\frac{|x|^{2}}{4t}}u(t,x)$, it follows from (\ref{P1}) that
\[|P_{t}(|u|^{\alpha}u)|=2|t||\nabla(e^{-i\frac{|x|^{2}}{4t}}|u|^{\alpha}u)|=2|t||\nabla(|v|^{\alpha}v)|.\]
From the above identity, (\ref{P1}) and H\"{o}lder's inequality, it follows that for any $1\leq p,q,r\leq \infty$ such that $\frac{1}{r}=\frac{\alpha}{p}+\frac{1}{q}$,
\begin{equation}\label{P2}
    \|P_{t}(|u|^{\alpha}u)\|_{L^{r}}\leq C|t|\|v\|_{L^{p}}^{\alpha}\|\nabla v\|_{L^{q}}\leq C\|u\|_{L^{p}}^{\alpha}\|P_{t}u\|_{L^{q}}.
\end{equation}

\subsection{Applications of the Pseudo-conformal Transformation}
We will investigate the scattering problem for (\ref{eq1}) by applying the pseudo-conformal transformation(see Chapter 7 in \cite{C2} for a review). By Remark \ref{rem1}, we can mainly concern about positively global solution $u(t)$ defined on $(0,+\infty)$, the scattering problem for $t\rightarrow -\infty$ can be treated similarly.

we consider the variables $(t,x)\in \mathbb{R}\times \mathbb{R}^{d}$ defined by
\begin{equation}\label{Variables PC-Trans}
    t=\frac{s}{1-s}, \,\,\, x=\frac{y}{1-s}, \,\,\, or \,\,\, equivalently, \,\,\, s=\frac{t}{1+t}, \,\,\, y=\frac{x}{1+t}.
\end{equation}
Given $0\leq a<b\leq \infty$ and $u$ defined on $(a,b)\times \mathbb{R}^{d}$, we set
\begin{equation}\label{PC-Trans}
    v(s,y)=(1-s)^{-\frac{d}{2}}u(\frac{s}{1-s},\frac{y}{1-s})e^{-i\frac{|y|^{2}}{4(1-s)}}
    =(1+t)^{\frac{d}{2}}u(t,x)e^{-i\frac{|x|^{2}}{4(1+t)}}
\end{equation}
for $y\in \mathbb{R}^{d}$ and $\frac{a}{1+a}<s<\frac{b}{1+b}$. In particular, if $u$ is defined on $(0,\infty)$, then $v$ is defined on $(0,1)$. One easily verifies that $u\in C([a,b],\Sigma)$ if and only if $v\in C([\frac{a}{1+a},\frac{b}{1+b}],\Sigma)$($0\leq a<b<\infty$ are given).

Furthermore, a straightforward calculation shows that $u$ satisfies (\ref{eq1}) on $(a,b)$ if and only if $v$ satisfies the nonautonomous Cauchy problem
\begin{equation}\label{nonautonomous}
    \left\{
  \begin{array}{ll}
    i \partial_{s}v+ \Delta_{y} v+\lambda(1-s)^{\frac{\alpha d-4}{2}}|v|^{\alpha}v=0, \,\,\, s>0, \, y\in \mathbb{R}^{d}\\
    v(0,y)=v_{0}(y)=u_{0}(x)e^{-i\frac{|x|^{2}}{4}}\in \Sigma, \,\,\, y\in \mathbb{R}^{d}
  \end{array}
\right.
\end{equation}
on the interval $(\frac{a}{1+a},\frac{b}{1+b})$. Note that the term $(1-s)^{\frac{\alpha d-4}{2}}$ is regular, except possibly at $t=1$, where it is singular for $\alpha<\frac{4}{d}$. Moreover, the following identities hold:
\begin{equation}\label{PC-T1}
    \|v(s)\|_{L^{\beta+2}}^{\beta+2}=(1+t)^{\frac{\beta d}{2}}\|u(t)\|_{L^{\beta+2}}^{\beta+2}, \,\,\, for \,\,\, \beta\geq 0,
\end{equation}
\begin{equation}\label{PC-T2}
    \|\nabla v(s)\|_{L^{2}}^{2}=\frac{1}{4}\|(x+2i(1+t)\nabla)u(t)\|_{L^{2}}^{2},
\end{equation}
\begin{equation}\label{PC-T3}
    \|\nabla u(t)\|_{L^{2}}^{2}=\frac{1}{4}\|(y-2i(1-s)\nabla)v(s)\|_{L^{2}}^{2}.
\end{equation}
Hence if we set
\[E_{1}(s)=\frac{1}{2}\|\nabla v(s)\|_{L^{2}}^{2}-(1-s)^{\frac{\alpha d-4}{2}}\frac{\lambda}{\alpha+2}\|v(s)\|_{L^{\alpha+2}}^{\alpha+2},\]
\[E_{2}(s)=(1-s)^{\frac{4-\alpha d}{2}}\frac{1}{2}\|\nabla v(s)\|_{L^{2}}^{2}-\frac{\lambda}{\alpha+2}\|v(s)\|_{L^{\alpha+2}}^{\alpha+2},\]
it follows from the the pseudo-conformal conservation law for (\ref{eq1}) that
\begin{equation}\label{E1}
    \frac{d}{ds}E_{1}(s)=-(1-s)^{\frac{\alpha d-6}{2}}\frac{4-\alpha d}{2}\frac{\lambda}{\alpha+2}\|v(s)\|_{L^{\alpha+2}}^{\alpha+2},
\end{equation}
\begin{equation}\label{E2}
    \frac{d}{ds}E_{2}(s)=(1-s)^{\frac{2-\alpha d}{2}}\frac{\alpha d-4}{4}\|\nabla v(s)\|_{L^{2}}^{2}.
\end{equation}
By applying the fixed point theorem in an appropriate space to the following equivalent integral equation for the nonautonomous equation (\ref{nonautonomous})
\begin{equation}\label{nonauto Integral Eq}
    v(s)=e^{is\Delta}v_{0}+i\lambda\int_{0}^{s}e^{i(s-\tau)\Delta}(1-\tau)^{\frac{\alpha d-4}{2}}|v(\tau)|^{\alpha}v(\tau)d\tau,
\end{equation}
Cazenave and Weissler obtained the following local well-posedness result for (\ref{nonautonomous}) in \cite{C1}(see \cite{C2} or \cite{C1}, Theorem 3.4 for the proof).
\begin{thm}\label{local existence}
Assume that $\lambda\in\mathbb{R}$,
\begin{equation}\label{low-energy-scatter-index}
    \frac{4}{d+2}<\alpha<\frac{4}{d-2} \,\,\, (2<\alpha<\infty, \,\,\, if \,\,\,d=1).
\end{equation}
It follows that for every $s_{0}\in\mathbb{R}$ and $\psi\in\Sigma$, there exist $S_{m}(s_{0},\psi)<s_{0}<S_{M}(s_{0},\psi)$ and a unique, maximal solution $v\in C((S_{m},S_{M}),\Sigma)$ of equation (\ref{nonautonomous}). The solution $v$ is maximal in the sense that if $S_{M}<\infty$ (respectively, $S_{m}>-\infty$), then $\|v(s)\|_{H^{1}}\rightarrow\infty$ as $s\uparrow S_{M}$ (respectively, $s\downarrow S_{m}$). In addition, if $S_{M}=1$, then $\liminf_{s\uparrow1}\{(1-s)^{\delta}\|v(s)\|_{H^{1}}\}>0$ with $\delta=\frac{d+2}{4}-\frac{1}{\alpha}$ if $d\geq3$, $\delta<1-\frac{1}{\alpha}$ if $d=2$, and $\delta=\frac{1}{2}-\frac{1}{\alpha}$ if $d=1$.
\end{thm}
The following useful observation indicates the inseparable relationship between the asymptotic behavior of $u(t)$ as $t\rightarrow+\infty$ and $v(s)$ as $s\rightarrow 1$(refer to \cite{C1}, Proposition 3.14 for the proof).
\begin{prop}\label{equivalent}
Let $u\in C([0,\infty),\Sigma)$ be a solution of equation (\ref{eq1}) and let $v\in C([0,1),\Sigma)$ be the corresponding solution of (\ref{nonautonomous}) defined by (\ref{PC-Trans}). It follows that $e^{-it\Delta}u(t)$ has a strong limit in $\Sigma$ as $t\rightarrow\infty$ if and only if $v(s)$ has a strong limit in $\Sigma$ as $s\uparrow1$, in which case
\begin{equation}\label{equivalent-scattering}
    \lim_{t\rightarrow\infty}e^{-it\Delta}u(t)=e^{i\frac{|y|^{2}}{4}}e^{-i\Delta}v(1) \,\,\,\,\,\, in \,\, \,\Sigma.
\end{equation}
\end{prop}

\section{scattering for small initial data}
In this section we consider the scattering theory in $\Sigma$ for small initial values $u_{0}$ under the assumption that $\lambda\in\mathbb{R}\setminus\{0\}$, $\alpha\leq\frac{4}{d+2}$. Note that if $\lambda>0$, $\alpha<\frac{4}{d+2}$, there are initial values $u_{0}\in\Sigma$ with arbitrary small norm $\|u_{0}\|_{\Sigma}$ that do not have a scattering state, even in the sense of $L^{2}(\mathbb{R}^{d})$(see \cite{C1}). We obtain the following result.
\begin{thm}\label{small data}
Assume $d\geq3$, $\lambda\in\mathbb{R}\setminus\{0\}$, $\alpha_{d}<\alpha<\frac{4}{d-2}$, where $\alpha_{d}=\frac{-d+\sqrt{d^{2}+16d}}{2d}$. Then there exists $\varepsilon_{0}>0$ with the following property. Let $u_{0}\in H^{2} \cap \mathcal{F}(H^{2}) \subset \Sigma$, $v_{0}=e^{-i\frac{|x|^{2}}{4}}u_{0}$ and let $u$ be the corresponding maximal solution of (\ref{eq1}). If $\|v_{0}\|_{H^{2}}\leq\varepsilon_{0}$(assuming further $\|u_{0}\|_{H^{1}}\leq\varepsilon_{0}$ when $\lambda>0$, $\alpha\geq \frac{4}{d}$), then the solution $u$ is global and scatters as $t\rightarrow \pm\infty$.
\end{thm}
\begin{proof}
For the prove of the scattering properties, we deal with only the positive time $t\rightarrow +\infty$, since $t\rightarrow-\infty$ can be treated in the same way. Under the assumptions of Theorem \ref{small data}, it is well known that there exists $\varepsilon_{0}>0$ such that the solution $u$ of Cauchy problem (\ref{eq1}) is global and bounded in $H^{1}(\mathbb{R}^{d})$, moreover, $u\in C(\mathbb{R},\Sigma)$. Thus we have $v(s,y)$ (the Pseudo-conformal Transformation of $u(t,x)$, see Section 2 for a review) defined by (\ref{PC-Trans}) satisfies the following nonautonomous integral equation
 \begin{equation}\label{nonauto Integral Eq1}
    v(s)=e^{is\Delta}v_{0}+i\lambda\int_{0}^{s}e^{i(s-\tau)\Delta}(1-\tau)^{\frac{\alpha d-4}{2}}|v(\tau)|^{\alpha}v(\tau)d\tau
\end{equation}
on the interval $(0,1)$, and $v\in C([0,1),\Sigma)$. \\
Let $(\gamma,\rho)$ be the admissible pair defined by
\begin{equation}\label{pair1}
    \gamma=\frac{4(\alpha+2)}{\alpha(d-2)}, \,\,\,\,\,\, \rho=\frac{d(\alpha+2)}{d+\alpha},
\end{equation}
and index $\rho^{\ast}=\frac{d(\alpha+2)}{d-2}$. One easily verifies that $\frac{1}{\rho'}=\frac{\alpha}{\rho^{\ast}}+\frac{1}{\rho}$ and $L^{\rho^{\ast}}(\mathbb{R}^{d})\hookrightarrow \dot{W}^{1,\rho}(\mathbb{R}^{d})$. Therefore, by applying the dispersive estimates(see \cite{C2}) and H\"{o}lder's inequality to the integral equation (\ref{nonauto Integral Eq1}), we have
\begin{equation*}
    \begin{array}{ll}
    \|\nabla v(s)\|_{L^{\rho}} \leq \|\nabla(e^{is\Delta}v_{0})\|_{L^{\rho}}+C\int_{0}^{s}(1-\tau)^{\frac{\alpha d-4}{2}}(s-\tau)^{-\frac{2}{\gamma}}\|v(\tau)\|_{L^{\rho^{\ast}}}^{\alpha}\|\nabla v(\tau)\|_{L^{\rho}}d\tau \\
    \,\,\,\,\,\,\,\,\,\,\,\,\,\,\,\,\,\,\,\,\,\,\,\,\,\,\,\,\, \leq C\|v_{0}\|_{H^{2}}+C\int_{0}^{s}(1-\tau)^{\frac{\alpha d-4}{2}}(s-\tau)^{-\frac{2}{\gamma}}\|\nabla v(\tau)\|_{L^{\rho}}^{\alpha+1}d\tau.
    \end{array}
\end{equation*}
Note that $\alpha>\alpha_{d}$, we have $\frac{4-\alpha d}{2}+\frac{2}{\gamma}<1$, thus we can deduce from H\"{o}lder's inequality that
\begin{equation}
    \int_{0}^{s}(1-\tau)^{\frac{\alpha d-4}{2}}(s-\tau)^{-\frac{2}{\gamma}}d\tau\leq C(\alpha,d).
\end{equation}
Set $\Theta(s)=\sup_{\tau\in[0,s]}\|\nabla v(\tau)\|_{L^{\rho}}$, for $0<s<1$. Then we can deduce from the above two estimates immediately that
\begin{equation}\label{continuation argument1}
    \Theta(s)\leq C\|v_{0}\|_{H^{2}}+C\Theta(s)^{\alpha+1} \,\,\, for \,\,\, all \,\,\, 0<s<1.
\end{equation}
Note that $u_{0}\in H^{2}\cap\mathcal{F}(H^{2})$, one easily verifies that $v_{0}\in H^{2}$ and $v\in C([0,1),H^{2})$, thus we have $\Theta\in C([0,1))$ and
\begin{equation}\label{continuation argument2}
    \lim_{s\rightarrow 0}\Theta(s)=\|\nabla v_{0}\|_{L^{\rho}}\leq C\|v_{0}\|_{H^{2}}.
\end{equation}
Applying (\ref{continuation argument2}), we deduce easily that if $\|v_{0}\|_{H^{2}}\leq \varepsilon_{0}$ where $\varepsilon_{0}>0$ is sufficiently small so that $(2C\varepsilon_{0})^{\alpha+1}<\varepsilon_{0}$, then
\[\Theta(s)\leq 2C\|v_{0}\|_{H^{2}} \,\,\,\,\, for \,\,\, all \,\,\,\, 0<s<1.\]
Letting $s\uparrow 1$, we deduce in particular that
\begin{equation}\label{decay estimate1}
    \sup_{s\in [0,1)}\|v(s)\|_{L^{\rho^{\ast}}}\leq C\sup_{s\in [0,1)} \|\nabla v(s)\|_{L^{\rho}}<\infty.
\end{equation}
Therefore we deduce from identity (\ref{PC-T1}) the following decay estimate for $u(t,x)$:
\begin{equation}\label{decay estimate2}
    \|u(t)\|_{L^{\rho^{\ast}}}\leq C(1+t)^{-d(\frac{1}{2}-\frac{1}{\rho^{\ast}})} \,\,\,\, for \,\,\,\, all \,\,\, t\geq 0.
\end{equation}
Therefore, it follows from Strichartz's estimates that for every $t\geq T\geq 0$,
\begin{eqnarray}\label{Stri1}
\nonumber
  \|u\|_{L^{\gamma}((0,t),W^{1,\rho})}&\leq&C\|u_{0}\|_{H^{1}}+C(\int_{0}^{T}\|u(\tau)\|_{L^{\rho^{\ast}}}^{\frac{\alpha\gamma}{\gamma-2}}d\tau)^{\frac{\gamma-2}{\gamma}}
\|u\|_{L^{\gamma}((0,T),W^{1,\rho})} \\ &+&C(\int_{T}^{t}\|u(\tau)\|_{L^{\rho^{\ast}}}^{\frac{\alpha\gamma}{\gamma-2}}d\tau)^{\frac{\gamma-2}{\gamma}}
\|u\|_{L^{\gamma}((T,t),W^{1,\rho})}.
\end{eqnarray}
Using (\ref{decay estimate2}), we get
\[\|u(\tau)\|_{L^{\rho^{\ast}}}^{\frac{\alpha\gamma}{\gamma-2}}\leq C(1+\tau)^{-\frac{\alpha d \gamma-8}{2(\gamma-2)}}.\]
Note that since $\alpha>\alpha_{d}=\frac{-d+\sqrt{d^{2}+16d}}{2d}$, we have $\alpha d \gamma-8>2(\gamma-2)$. Therefore for $T$ large enough,
\begin{equation}\label{Stri2}
    C(\int_{T}^{t}\|u(\tau)\|_{L^{\rho^{\ast}}}^{\frac{\alpha\gamma}{\gamma-2}}d\tau)^{\frac{\gamma-2}{\gamma}}\leq \frac{1}{2}.
\end{equation}
On the other hand, $u\in L^{\infty}((0,T),H^{1}(\mathbb{R}^{d}))\cap L^{q}((0,T),W^{1,r}(\mathbb{R}^{d}))$. Therefore, it follows from (\ref{Stri1}) and (\ref{Stri2}) that
\[\|u\|_{L^{\gamma}((0,t),W^{1,\rho})}\leq C+\frac{1}{2}\|u\|_{L^{\gamma}((0,t),W^{1,\rho})}.\]
Letting $t\uparrow \infty$, we obtain
\[\|u\|_{L^{\gamma}((0,\infty),W^{1,\rho})}\leq 2C.\]
This also implies that $|u|^{\alpha}u\in L^{\gamma'}((0,\infty),W^{1,\rho'}(\mathbb{R}^{d}))$. Applying again Strichartz's estimates, one obtains the result for every admissible pair. Let $P_{t}u=(x+2it\nabla)u$, by applying Strichartz's estimates, we obtain
\begin{eqnarray}\label{Stri3}
\nonumber
  \|P_{t}u\|_{L^{\gamma}((0,t),L^{\rho})}&\leq&C\|xu_{0}\|_{L^{2}}+C(\int_{0}^{T}\|u(\tau)\|_{L^{\rho^{\ast}}}^{\frac{\alpha\gamma}{\gamma-2}}d\tau)
  ^{\frac{\gamma-2}{\gamma}}\|P_{t}u\|_{L^{\gamma}((0,T),L^{\rho})} \\ &+&C(\int_{T}^{t}\|u(\tau)\|_{L^{\rho^{\ast}}}^{\frac{\alpha\gamma}{\gamma-2}}d\tau)^{\frac{\gamma-2}{\gamma}}
\|P_{t}u\|_{L^{\gamma}((T,t),L^{\rho})}
\end{eqnarray}
for every $0\leq T\leq t$. Then one concludes similarly as above that
\[\|(x+2it\nabla)u\|_{L^{\gamma}((0,\infty),L^{\rho})}\leq 2C,\]
and $P(|u|^{\alpha}u)\in L^{\gamma'}((0,\infty),L^{\rho'}(\mathbb{R}^{d}))$. Therefore for $0<t<s$, by Strichartz's estimates, we have
\[\|e^{-it\Delta}u(t)-e^{-is\Delta}u(s)\|_{H^{1}}\leq C\||u|^{\alpha}u\|_{L^{\gamma'}((t,s),W^{1,\rho'})},\]
\[\|x(e^{-it\Delta}u(t)-e^{-is\Delta}u(s))\|_{L^{2}}\leq C\|(x+2i\tau\nabla)|u|^{\alpha}u\|_{L^{\gamma'}((t,s),L^{\rho'})}.\]
Thus, we get immediately
\[\|e^{-it\Delta}u(t)-e^{-is\Delta}u(s)\|_{H^{1}}\rightarrow0,\]
\[\|x(e^{-it\Delta}u(t)-e^{-is\Delta}u(s))\|_{L^{2}}\rightarrow0,\]
as $t,s\rightarrow\infty$. Hence there exists $u^{+}\in\Sigma$ such that $e^{-it\Delta}u(t)\rightarrow u^{+}$ in $\Sigma$ as $t\rightarrow\infty$.
\end{proof}
\begin{rem}
Since $\|v_{0}\|_{H^{2}}\leq C\|u_{0}\|_{H^{2}\cap\mathcal{F}(H^{2})}$, the smallness assumptions in Theorem \ref{small data} can be deduced from assuming that $\|u_{0}\|_{H^{2}\cap\mathcal{F}(H^{2})}$ is small. Note that $\max\{\frac{2}{d},\frac{4}{d+4}\}<\alpha_{d}<\frac{4}{d+2}$, we claim that for $\lambda>0$, there exists a constant $K>0$, such that for any $C_{0}\geq K$, there exists a initial data $u_{0}$ satisfying $\|u_{0}\|_{H^{2}\cap\mathcal{F}(H^{2})}=C_{0}$, which do not have a scattering state, even in the sense of $L^{2}(\mathbb{R}^{d})$. To see this, let $\varphi\in \Sigma$ be a nontrivial solution of the equation
\[-\Delta\varphi+\varphi=\lambda|\varphi|^{\alpha}\varphi.\]
Given $\omega>0$, set $\varphi_{\omega}(x)=\omega^{\frac{1}{\alpha}}\varphi(x\sqrt{\omega})$. It follows that $u_{\omega}(t,x)=e^{i\omega t}\varphi_{\omega}(x)$ satisfies (\ref{eq1}) and does not have any strong limit as $t\rightarrow\pm\infty$ in $L^{2}(\mathbb{R}^{d})$. One easily verifies that if $\alpha>\alpha_{d}$, then $\|\varphi_{\omega}\|_{H^{2}\cap\mathcal{F}(H^{2})}\rightarrow \infty$ as $\omega\rightarrow0$ and $\omega\rightarrow\infty$, which indicates that $K=\inf_{\omega\in(0,\infty)}\|\varphi_{\omega}\|_{H^{2}\cap\mathcal{F}(H^{2})}$ is attained. This proves our claim. Furthermore, if $\alpha<\frac{4}{d+4}$, since $\|\varphi_{\omega}\|_{H^{2}\cap\mathcal{F}(H^{2})}\rightarrow0$ as $\omega\rightarrow0$, there are arbitrary small initial values $u_{0}\in H^{2} \cap \mathcal{F}(H^{2})$ that do not have a scattering state.
\end{rem}

\section{Lower bound estimates for the singular solutions of nonautonomous equation}
In this section we will present a lower bound estimates for the $L^{\alpha+2}$ norm of the global solutions $u(t,x)$ to Cauchy problem (\ref{eq1}) that do not scatter as $t\rightarrow\pm\infty$ in the case $\lambda>0$. Therefore, we could deduce from these results that the global solution to (\ref{eq1}) with fast decay must scatter at $\pm\infty$, we will apply the following Propositions in this Section to the investigation on ``rapidly decaying solutions" in Section 5.
\begin{lem}\label{blowup0}
Assume $\lambda\in \mathbb{R}\setminus\{0\}$, $\frac{4}{d+2}<\alpha<\frac{4}{d-2}$($2<\alpha<\infty$, if $d=1$). Let $v_{0}\in\Sigma$ and $v\in C((S_{m},S_{M}),\Sigma)$ be the corresponding maximal solution of (\ref{nonautonomous}) given by Theorem \ref{local existence}. Then if $S_{M}(0,v_{0})=1$, we have the following lower estimate of the blowup rate
\begin{equation}\label{Lower bound0}
    \|\nabla v(s)\|_{L^{2}}^{2}\geq C(1-s)^{-2\theta}
\end{equation}
for some constant $C>0$ and all $s\in[0,1)$, where $\theta=\frac{d+2}{4}-\frac{1}{\alpha}$ if $d\geq3$, $\theta$ any positive number less than $1-\frac{1}{\alpha}$ if $d=2$, and $\theta=1-\frac{2}{\alpha}$ if $d=1$.
\end{lem}
\begin{proof}
By Theorem \ref{local existence}, Lemma \ref{Lower bound0} holds for $d\geq 2$, therefore we need only consider $d=1$.
Fix $s_{0}\in[0,1)$. Let $f(s)=\lambda(1-s)^{\frac{\alpha-4}{2}}$, for $0\leq s<1$. It follows from equation (\ref{nonautonomous}) and Strichartz's estimates that
\begin{equation}\label{eq3}
   \|v\|_{L^{\infty}((s_{0},s),H^{1})}\leq C\|v(s_{0})\|_{H^{1}}+C\|f|v|^{\alpha}v\|_{L^{1}((s_{0},s),H^{1})}
\end{equation}
for all $s\in(s_{0},1)$. On the other hand,
\begin{equation}\label{eq4}
    \||v|^{\alpha}v\|_{H^{1}}\leq C\|v\|_{L^{\infty}}^{\alpha}\|v\|_{H^{1}}
\end{equation}
and, by Gagliardo-Nirenberg's inequality,
\begin{equation}\label{eq5}
    \|v\|_{L^{\infty}}\leq C\|v\|_{H^{1}}^{1/2}\|v\|_{L^{2}}^{1/2}\leq C\|v\|_{H^{1}}^{1/2}.
\end{equation}
Therefore, we deduce from the above three inequalities (\ref{eq3}), (\ref{eq4}) and (\ref{eq5})that there exists a constant $K>0$ independent of $s_{0}$ and $s$ such that
\begin{equation}\label{eq6}
    \|v\|_{L^{\infty}((s_{0},s),H^{1})}\leq K\|v(s_{0})\|_{H^{1}}+K\|f\|_{L^{1}(s_{0},s)}\|v\|_{L^{\infty}((s_{0},s),H^{1})}^{1+\frac{\alpha}{2}}.
\end{equation}
Now, since by Theorem \ref{local existence} we have
 \[\limsup_{s\uparrow1}\|v(s)\|_{H^{1}}=\infty,\]
there exists $s_{1}\in(s_{0},1)$ such that $\|v\|_{L^{\infty}((s_{0},s_{1}),H^{1})}=(K+1)\|v(s_{0})\|_{H^{1}}$. Letting $s=s_{1}$ in (\ref{eq6}), we obtain
\[\|v(s_{0})\|_{H^{1}}\leq K((K+1)\|v(s_{0})\|_{H^{1}})^{1+\frac{\alpha}{2}}\|f\|_{L^{1}(s_{0},s_{1})},\]
hence
\[1\leq CK(K+1)^{1+\frac{\alpha}{2}}\|v(s_{0})\|_{H^{1}}^{\frac{\alpha}{2}}(1-s_{0})^{\frac{\alpha-2}{2}}.\]
Since $s_{0}\in[0,1)$ is arbitrary, we obtain for $d=1$ that
\begin{equation}\label{eq7}
    \|\nabla v(s)\|_{L^{2}}\geq C(1-s)^{-\frac{\alpha-2}{\alpha}}
\end{equation}
for some constant $C>0$ and all $s\in[0,1)$. This closes our proof.
\end{proof}
\begin{prop}\label{blowup1}
Assume $\lambda>0$, $\frac{4}{d+2}<\alpha\leq\frac{4}{d}$($2<\alpha\leq4$, if $d=1$). Let $u_{0}\in\Sigma$ and $u$ be the corresponding maximal solution of (\ref{eq1}), then if $u$ is positively(resp. negatively) global and doesn't scatter at $+\infty$(resp. $-\infty$), we have
\begin{equation}\label{Lower bound1}
    \|u(t)\|_{L^{\alpha+2}}\geq C(1+|t|)^{-\frac{2(1-\theta)}{\alpha+2}},
\end{equation}
for all $t\in(0,+\infty)$(resp. $t\in(-\infty,0)$), where $\theta$ is defined the same as in Lemma \ref{blowup0}. Moreover, for $d\geq 3$, $\alpha=\alpha(d)=\frac{2-d+\sqrt{d^{2}+12d+4}}{2d}$, we can derive a better lower estimate
\begin{equation}\label{Lower bound2}
    \|u(t)\|_{L^{\alpha+2}}\geq C(1+|t|)^{-\frac{\alpha d}{2(\alpha+2)}}[\log(1+|t|)]^{\frac{1}{\alpha+2}},
\end{equation}
for all $t\in(0,+\infty)$(resp. $t\in(-\infty,0)$).
\end{prop}
\begin{proof}
We will only deal with the positive time, since $(-\infty,0)$ can be treated similarly. By Proposition \ref{equivalent}, we deduce from the assumption positively global solution $u$ doesn't have scattering state $u^{+}$ at $+\infty$ that the nonautonomous equation (\ref{nonautonomous}) blows up at $s=1$(i.e., $S_{M}(0,v_{0})=1$). Hence by Theorem \ref{local existence}, one has
\begin{equation}\label{blowup}
   \limsup_{s\uparrow1}\|v(s)\|_{H^{1}}=\infty.
\end{equation}
Furthermore, by Lemma \ref{blowup0}, we get
\begin{equation}\label{blowup rate}
    \|\nabla v(s)\|_{L^{2}}^{2}\geq C(1-s)^{-2\theta}
\end{equation}
for some constant $C>0$ and all $s\in[0,1)$.
Note that $\lambda>0$, $\alpha\leq\frac{4}{d}$, we deduce from (\ref{E1}) that $\frac{d}{ds}E_{1}(s)\leq 0$, and so
\begin{equation}\label{eq8}
    \frac{1}{2}\|\nabla v(s)\|_{L^{2}}^{2}\leq E_{1}(0)+\frac{\lambda}{\alpha+2}(1-s)^{\frac{\alpha d-4}{2}}\|v(s)\|_{L^{\alpha+2}}^{\alpha+2}.
\end{equation}
Note that $\lambda>0$, $\frac{4}{d+2}<\alpha\leq\frac{4}{d}$($2<\alpha\leq4$, if $d=1$), from (\ref{blowup rate}) and (\ref{eq8}) we infer that
\[\|v(s)\|_{L^{\alpha+2}}^{\alpha+2}\geq C(1-s)^{-2\theta+\frac{4-\alpha d}{2}}\]
for some constant $C>0$ and all $s\in[0,1)$. Thus it follows from identity (\ref{Variables PC-Trans}) and (\ref{PC-T1}) that
\begin{equation}\label{eq9}
    \|u(t)\|_{L^{\alpha+2}}\geq C(1+t)^{-\frac{2(1-\theta)}{\alpha+2}} \,\,\,\, for \,\,\, all \,\,\, t\in[0,\infty).
\end{equation}
This indicates that the lower estimate (\ref{Lower bound1}) holds for positive time $(0,+\infty)$.
Furthermore, if $d\geq3$, $\alpha=\alpha(d)$, we deduce from (\ref{E2}) and (\ref{blowup rate}) that
\begin{eqnarray}
\nonumber &&\frac{\lambda}{\alpha+2}\|v(s)\|_{L^{\alpha+2}}^{\alpha+2}+E_{2}(0) \\
\nonumber &=&\frac{1}{2}(1-s)^{\frac{4-\alpha d}{2}}\|\nabla v(s)\|_{L^{2}}^{2}+\frac{4-\alpha d}{4}\int_{0}^{s}(1-\tau)^{\frac{2-\alpha d}{2}}\|\nabla v(\tau)\|_{L^{2}}^{2}d\tau \\
\nonumber &\geq& C+C\int_{0}^{s}(1-\tau)^{-1}d\tau \geq C+C\log(\frac{1}{1-s}),
\end{eqnarray}
from which we get immediately
\[\|v(s)\|_{L^{\alpha+2}}^{\alpha+2}\geq C\log(\frac{1}{1-s})\]
for some constant $C>0$ and all $s\in[0,1)$. Thus it follows from identity (\ref{Variables PC-Trans}) and (\ref{PC-T1}) that
\begin{equation}\label{eq10}
\|u(t)\|_{L^{\alpha+2}}\geq C(1+t)^{-\frac{\alpha d}{2(\alpha+2)}}[\log(1+t)]^{\frac{1}{\alpha+2}}
\end{equation}
for all $t\in[0,\infty)$. This closes the proof of Proposition \ref{blowup1} for positive time $(0,+\infty)$, and the arguments for $(-\infty,0)$ being similar.
\end{proof}
When $\alpha>\frac{4}{d}$, there is an lower estimate of integral form.
\begin{prop}\label{blowup2}
Assume $\lambda>0$, $\frac{4}{d}<\alpha<\frac{4}{d-2}$($\frac{4}{d}<\alpha<\infty$, if $d=1,2$). Let $u_{0}\in\Sigma$ and $u$ be the corresponding maximal solution of (\ref{eq1}), then if $u$ is positively(resp. negatively) global and doesn't scatter at $+\infty$(resp. $-\infty$), then there exists a $t_{0}>0$ such that
\begin{equation}\label{Lower bound3}
    |\int_{0}^{t}(1+|\tau|)\|u(\tau)\|_{L^{\alpha+2}}^{\alpha+2}d\tau|\geq C(1+|t|)^{2\theta},
\end{equation}
for all $t\in[t_{0},+\infty)$(resp. $t\in(-\infty,-t_{0}]$), where $\theta$ is defined the same as in Lemma \ref{blowup0}.
\end{prop}
\begin{proof}
We will only deal with the positive time, since $(-\infty,0)$ can be treated similarly. By Proposition \ref{equivalent}, we deduce from the assumption positively global solution $u$ doesn't have scattering state $u^{+}$ at $+\infty$ that the nonautonomous equation (\ref{nonautonomous}) blows up at $s=1$(i.e., $S_{M}(0,v_{0})=1$). Hence by Theorem \ref{local existence}, one has
\begin{equation}\label{blowup}
   \limsup_{s\uparrow1}\|v(s)\|_{H^{1}}=\infty.
\end{equation}
Moreover, by Lemma \ref{blowup0}, we get
\begin{equation}\label{blowup rate1}
    \|\nabla v(s)\|_{L^{2}}^{2}\geq C(1-s)^{-2\theta}
\end{equation}
for some constant $C>0$ and all $s\in[0,1)$. From (\ref{E1}) we can deduce the following identity
\begin{eqnarray}\label{eq11}
   \frac{1}{2}\|\nabla v(s)\|_{L^{2}}^{2}=E_{1}(0)+\frac{\lambda}{\alpha+2}(1-s)^{\frac{\alpha d-4}{2}}\|v(s)\|_{L^{\alpha+2}}^{\alpha+2} \\
   \nonumber +\frac{\lambda(\alpha d-4)}{2(\alpha+2)}\int_{0}^{s}(1-\tau)^{\frac{\alpha d-6}{2}}\|v(\tau)\|_{L^{\alpha+2}}^{\alpha+2}d\tau.&&
\end{eqnarray}
Let $h(s)=\frac{\lambda}{\alpha+2}\int_{0}^{s}(1-\tau)^{\frac{\alpha d-6}{2}}\|v(\tau)\|_{L^{\alpha+2}}^{\alpha+2}d\tau$, using (\ref{blowup rate1}) and (\ref{eq11}) we get
\begin{equation}\label{eq12}
    (1-s)^{\frac{\alpha d-2}{2}}((1-s)^{\frac{4-\alpha d}{2}}h(s))'_{s}\geq C(1-s)^{-2\theta}
\end{equation}
for some constant $C>0$ and all $s\in[0,1)$. Note that $\alpha>\frac{4}{d}$ and $h(0)=0$, by integrating the above inequality (\ref{eq12}), we infer that
\begin{equation}\label{eq14}
    (1-s)^{\frac{4-\alpha d}{2}}h(s)\geq C\{(1-s)^{-2\theta+\frac{4-\alpha d}{2}}-1\} \,\,\,\,\, for \,\,\, all \,\, s\in[0,1).
\end{equation}
Since $\alpha>\frac{4}{d}$, one easily verifies that $-2\theta+\frac{4-\alpha d}{2}<0$, thus there exists a $0<s_{0}<1$
such that
\begin{equation}\label{eq15}
    h(s)\geq C(1-s)^{-2\theta}
\end{equation}
for all $s\in[s_{0},1)$. Applying (\ref{Variables PC-Trans}) and (\ref{PC-T1}), (\ref{eq15}) yields
\begin{equation}\label{Lower bound3}
    \int_{0}^{t}(1+\tau)\|u(\tau)\|_{L^{\alpha+2}}^{\alpha+2}d\tau\geq C(1+t)^{2\theta},
\end{equation}
for all $t\in[\frac{s_{0}}{1-s_{0}},+\infty)$. This closes the proof of Proposition \ref{blowup2} for positive time $(0,+\infty)$, and the arguments for $(-\infty,0)$ being similar.
\end{proof}

\section{scattering theory for the focusing NLS}
As is well known, if $\alpha\geq\alpha(d)$, scattering theory holds in whole $\Sigma$ space in the defocusing case(see \cite{G1,N2,T3}), this section is devoted to the studying on the scattering theory for the focusing NLS.

For $\lambda\in \mathbb{R}$, $\alpha(d)<\alpha<\frac{4}{d-2}$($\alpha(d)<\alpha<\infty$, if $d=1,2$), let
\begin{equation}\label{definition}
    \|u\|_{X_{\infty}}=\sup_{0<t<\infty}t^{\beta}\|u(t)\|_{L^{\alpha+2}}, \,\,\,\,\,\, \|\varphi\|_{W_{\infty}}=\sup_{0<t<\infty}t^{\beta}\|e^{it\Delta}\varphi\|_{L^{\alpha+2}},
\end{equation}
where $\beta=\frac{4-(d-2)\alpha}{2\alpha(\alpha+2)}$. Cazenave and Weissler proved in \cite{C3} that there exists $\varepsilon_{0}>0$ with the following property. Let $u_{0}\in H^{1}(\mathbb{R}^{d})$ and let $u$ be the corresponding unique, maximal strong $H^{1}$ solution of (\ref{eq1}), if $\|u_{0}\|_{W_{\infty}}\leq\varepsilon\leq\varepsilon_{0}$, then the solution $u$ is positively global, and $\|u\|_{X_{\infty}}\leq 2\varepsilon$. The same conclusion also holds for $(-\infty,0)$ with the usual modification.

So it's natural to ask if $u_{0}\in\Sigma$ satisfying $\|u_{0}\|_{W_{\infty}}\leq\varepsilon_{0}$, does the corresponding global solution $u$ scatter in $\Sigma$? In the defocusing case, it's well known that the answer is yes; as to the focusing case, we will give a positive answer below.
\begin{thm}\label{XW}
Assume $\lambda>0$, $\alpha(d)<\alpha<\frac{4}{d-2}$ if $d\geq3$, and $4\leq\alpha<\infty$, if $d=1$. Let $u_{0}\in\Sigma$ and let $u$ be the corresponding unique, maximal solution of (\ref{eq1}). Then there exists $\varepsilon_{0}>0$ such that, if $\|u_{0}\|_{W_{\infty}}\leq\varepsilon_{0}$, then the solution $u$ is positively global and scatters at $+\infty$ in $\Sigma$. The same conclusion also holds for $(-\infty,0)$ with the usual modification on assumptions.
\end{thm}
\begin{proof}
We have known that(see \cite{C2,C3}) there exists a $\varepsilon_{0}>0$, such that if $\|u_{0}\|_{W_{\infty}}\leq\varepsilon_{0}$, then the solution $u$ is positively global, and
\begin{equation}\label{eq16}
\|u\|_{X_{\infty}}\leq 2\varepsilon_{0}.
\end{equation}
This property also holds for any $0<\varepsilon\leq\varepsilon_{0}$. We will prove the scattering properties by contradiction arguments below.

If $d\geq3$, $\alpha(d)<\alpha\leq\frac{4}{d}$, or $d=1$, $\alpha=4$, we deduce from (\ref{eq16}), (\ref{definition}) and Proposition \ref{blowup1} that, if $u$ doesn't scatter at $+\infty$, then
\begin{equation}\label{eq17}
    C_{1}(\alpha,d)\varepsilon_{0}(1+t)^{-\frac{4-(d-2)\alpha}{2\alpha}}\geq\|u(t)\|_{L^{\alpha+2}}^{\alpha+2}\geq C(1+t)^{-\frac{4-(d-2)\alpha}{2\alpha}}
\end{equation}
for some constants $C_{1},C>0$, and all $t\in[1,+\infty)$. It is absurd if we take $\varepsilon_{0}$ sufficiently small such that $C/C_{1}>\varepsilon_{0}$.

If $d\geq3$, $\frac{4}{d}<\alpha<\frac{4}{d-2}$, or $d=1$, $4<\alpha<\infty$, it follows from (\ref{eq16}), (\ref{definition}) and Proposition \ref{blowup2} that, if $u$ doesn't scatter at $+\infty$, then there exists a $t_{0}>0$ such that
\begin{equation}\label{eq18}
    C_{2}(\alpha,d)\varepsilon_{0}(1+t)^{2\theta}\geq\int_{0}^{t}(1+\tau)\|u(\tau)\|_{L^{\alpha+2}}^{\alpha+2}d\tau\geq C(1+t)^{2\theta}
\end{equation}
for some constants $C_{2},C>0$, and all $t\in[t_{0},+\infty)$, where $\theta$ is defined the same as in Lemma \ref{blowup0}. Therefore it is absurd if we take $\varepsilon_{0}$ sufficiently small such that $C/C_{2}>\varepsilon_{0}$. This closes the proof of Theorem \ref{XW} for positive time $(0,+\infty)$, and the arguments for $(-\infty,0)$ being similar.
\end{proof}
\begin{rem}\label{rem2}
Note that we can deduce from the isometric properties of $(e^{it\Delta})_{t\in\mathbb{R}}$ and dispersive estimate(refer to \cite{C2}) that
\[(1+|t|)^{\frac{\alpha d}{2(\alpha+2)}}\|e^{it\Delta}u_{0}\|_{L^{\alpha+2}}\leq C\|u_{0}\|_{H^{1}\cap L^{\frac{\alpha+2}{\alpha+1}}}.\]
Since $\beta<\frac{\alpha d}{2(\alpha+2)}$, we infer that $\|u_{0}\|_{W_{\infty}}\leq C\|u_{0}\|_{H^{1}\cap L^{\frac{\alpha+2}{\alpha+1}}}$. Thus by Theorem \ref{XW}, for $\alpha>\alpha(d)$, there exists $\varepsilon_{0}>0$ such that if
\[\|u_{0}\|_{H^{1}\cap L^{\frac{\alpha+2}{\alpha+1}}}\leq \varepsilon_{0},\]
then the corresponding solution $u$ is global and scattering states $u^{\pm}$ exist in $\Sigma$ at $\pm\infty$.
\end{rem}
The concept ``rapidly decaying solutions"(see Definition \ref{RDS}) plays an important role in the scattering theory for focusing NLS. When $\lambda>0$, $\alpha>\alpha(d)$, Cazenave and Weissler obtained the scattering results for ``rapidly decaying solutions", in terms of which they characterized the sets $\mathcal{R}_{\pm}$(refer to \cite{C1}). We will extend this work to the critical power $\alpha=\alpha(d)$ below.
\begin{thm}\label{rapidly decaying}
Assume $\lambda>0$, $\alpha=\alpha(d)=\frac{2-d+\sqrt{d^{2}+12d+4}}{2d}$ and $d\geq1$, $d\neq2$. Let $u_{0}\in\Sigma$ be such that the corresponding solution $u$ is positively(resp. negatively) global with rapid decay(see Definition \ref{RDS}), i.e.,
\begin{equation}\label{rapid decay}
    \|u\|_{L^{a,\infty}((0,\infty),L^{\alpha+2})}<\infty \,\,\,\,\, (resp. \,\,\,\, \|u\|_{L^{a,\infty}((-\infty,0),L^{\alpha+2})}<\infty),
\end{equation}
where $a=\frac{2\alpha(\alpha+2)}{4-\alpha(d-2)}$, then $u_{0}$ has scattering state at $+\infty$(resp. $-\infty$).
\end{thm}
\begin{proof}
By Remark \ref{rem1}, we need only deal with the positive time $t\rightarrow+\infty$. We consider separately the cases $d\geq3$ and $d=1$.

First we consider the simpler case $d\geq3$. We argue by contradiction and assume that $u$ doesn't scatter at $+\infty$. Thus by (\ref{rapid decay}), Proposition \ref{blowup1} and (\ref{Lorentz1}), we get immediately
\begin{eqnarray}
 \nonumber \infty&>&\|u\|_{L^{a,\infty}((0,\infty),L^{\alpha+2})}^{a}\geq C\|(1+t)^{-\frac{\alpha d}{2(\alpha+2)}}[\log(1+t)]^{\frac{1}{\alpha+2}}\|_{L^{a,\infty}(0,\infty)}^{a} \\
 \nonumber &\geq&C\|(1+t)^{-1}[\log(1+t)]^{\frac{2}{\alpha d}}\|_{L^{1,\infty}(0,\infty)},
\end{eqnarray}
which is absurd. This closes our proof for $d\geq3$.

Now we proceed to the case $d=1$. We argue by contradiction and assume that $u$ doesn't scatter at $+\infty$. Then by Proposition \ref{equivalent}, the nonautonomous equation (\ref{nonautonomous}) blows up at $s=1$(i.e., $S_{M}(0,v_{0})=1$). Hence by Theorem \ref{local existence}, one has
\begin{equation}\label{eq19}
   \lim_{s\uparrow1}\|v(s)\|_{H^{1}}=\infty.
\end{equation}
Since $\alpha=\alpha(d)<\frac{4}{d}$, one can deduce from the energy estimates that the positively global $u$ is bounded in $H^{1}(\mathbb{R}^{d})$, so by using (\ref{eq19}) and (\ref{PC-T2}), we get
\begin{equation}\label{eq20}
    \lim_{t\rightarrow+\infty}\|P_{t}u(t)\|_{L^{2}}=\lim_{t\rightarrow+\infty}\|(x+2it\nabla)u(t)\|_{L^{2}}=\infty.
\end{equation}
On the other hand, it follows from (\ref{P3}) and Strichartz's estimates that
\begin{equation}\label{eq21}
    \|P_{_{\tau}}u(\tau)\|_{L^{\infty}((t,T),L^{2})}\leq C\|P_{t}u(t)\|_{L^{2}}+C\|P_{\tau}(|u|^{\alpha}u(\tau))\|_{L^{1}((t,T),L^{2})}
\end{equation}
for any $0<t<T<\infty$. We define a multiplier $M_{t}$ by $M_{t}=e^{\frac{i|x|^{2}}{4t}}$. Then by Gagliardo-Nirenberg's inequality, combined with (\ref{P1}) and (\ref{P2}) we have for any $0<t<T<\infty$,
\begin{eqnarray}\label{nonlinearity}
  \nonumber \|P_{\tau}(|u|^{\alpha}u(\tau))\|_{L^{1}((t,T),L^{2})}&\leq&C \int_{t}^{T}\|u(\tau)\|_{L^{\infty}}^{\alpha}\|P_{\tau}u(\tau)\|_{L^{2}}d\tau \\
&\leq&C \int_{t}^{T}\|M_{-\tau}u\|_{L^{\alpha+2}}^{\frac{\alpha(\alpha+2)}{\alpha+4}}\|M_{-\tau}u\|_{\dot{H}^{1}}^{\frac{2\alpha}{\alpha+4}}
\|P_{\tau}u(\tau)\|_{L^{2}}d\tau \\
\nonumber &\leq&C \int_{t}^{T}\tau^{-\frac{2\alpha}{\alpha+4}}\|u(\tau)\|_{L^{\alpha+2}}^{\frac{\alpha(\alpha+2)}{\alpha+4}}
\|P_{\tau}u(\tau)\|_{L^{2}}^{\frac{3\alpha+4}{\alpha+4}}d\tau.
\end{eqnarray}
Therefore, we deduce from the generalized H\"{o}lder's inequality in Lorentz spaces(see Lemma \ref{Lorentz2}) that
\[\|P_{\tau}(|u|^{\alpha}u)\|_{L^{1}((t,T),L^{2})}\leq C\|u\|_{L^{a,\infty}((0,\infty),L^{\alpha+2})}^{\frac{\alpha(\alpha+2)}{\alpha+4}}\|\tau^{-\frac{2\alpha}{\alpha+4}}\|_{L^{2,1}(t,T)}
\|P_{\tau}u\|_{L^{\infty}((t,T),L^{2})}^{\frac{3\alpha+4}{\alpha+4}},\]
and hence by assumption (\ref{rapid decay}) and (\ref{eq21}) we obtain that there exists a constant $K$ independent of $t$ and $T$ such that
\begin{equation}\label{eq22}
    \|P_{_{\tau}}u\|_{L^{\infty}((t,T),L^{2})}\leq K\|P_{t}u(t)\|_{L^{2}}+K\|\tau^{-\frac{2\alpha}{\alpha+4}}\|_{L^{2,1}(t,T)}
    \|P_{\tau}u\|_{L^{\infty}((t,T),L^{2})}^{\frac{3\alpha+4}{\alpha+4}}.
\end{equation}
Note that $u\in C((0,\infty),\Sigma)$, so $\|P_{_{\tau}}u\|_{L^{\infty}((t,T),L^{2})}$is continuous and nondecreasing about $T$ on $(t,\infty)$ and note also that
\[\|P_{_{\tau}}u\|_{L^{\infty}((t,T),L^{2})}\rightarrow\|P_{t}u(t)\|_{L^{2}}, \,\,\,\,\, as \,\,\, T\downarrow t,\]
thus we deduce from (\ref{eq20}) that there exists $T_{0}\in(t,\infty)$ such that $\|P_{_{\tau}}u\|_{L^{\infty}((t,T_{0}),L^{2})}=(K+1)\|P_{t}u(t)\|_{L^{2}}$. Letting $T=T_{0}$ in $(\ref{eq22})$, we obtain
\begin{equation}\label{eq23}
    \|P_{t}u(t)\|_{L^{2}}\leq K((K+1)\|P_{t}u(t)\|_{L^{2}})^{\frac{3\alpha+4}{\alpha+4}}\|\tau^{-\frac{2\alpha}{\alpha+4}}\|_{L^{2,1}(t,T_{0})}.
\end{equation}
Since by (\ref{Lorentz0}) we have $\|\tau^{-\frac{2\alpha}{\alpha+4}}\|_{L^{2,1}(t,T_{0})}\leq \|\tau^{-\frac{2\alpha}{\alpha+4}}\|_{L^{2,1}(t,\infty)}\leq Ct^{-\frac{3\alpha-4}{2(\alpha+4)}}$, we deduce from (\ref{eq23}) that
\[1\leq CK(K+1)^{\frac{3\alpha+4}{\alpha+4}}\|P_{t}u(t)\|_{L^{2}}^{\frac{2\alpha}{\alpha+4}}t^{-\frac{3\alpha-4}{2(\alpha+4)}},\]
and hence
\begin{equation}\label{eq24}
    \|P_{t}u(t)\|_{L^{2}}=\|(x+2it\nabla)u(t)\|_{L^{2}}\geq Ct^{\frac{3}{4}-\frac{1}{\alpha}}
\end{equation}
for some constant $C>0$ and arbitrary $t\in[0,\infty)$.
Now taking the $L_{x}^{2}$ coupling of the equation (\ref{eq1}) and $it^{\frac{\alpha d}{2}}\Delta u+t^{\frac{\alpha d}{2}-1}x\cdot\nabla u-i\frac{|x|^{2}}{4}t^{\frac{\alpha d}{2}-2}u$ and taking the real part, we obtain
\begin{equation}\label{P-conservation}
    \partial_{t}N(t)=-(2-\frac{\alpha d}{2})t^{\frac{\alpha d}{2}-3}\|(x+2it\nabla)u(t)\|_{L^{2}}^{2},
\end{equation}
where
\begin{equation}\label{P-mass}
    N(t)=t^{\frac{\alpha d}{2}-2}\|(x+2it\nabla)u(t)\|_{L^{2}}^{2}-\frac{8\lambda}{\alpha+2}t^{\frac{\alpha d}{2}}\|u(t)\|_{L^{\alpha+2}}^{\alpha+2}.
\end{equation}
Integrating the identity (\ref{P-conservation}), we get
\begin{equation}\label{P-conservation1}
    t^{\frac{\alpha d-4}{2}}\|P_{t}u\|_{L^{2}}^{2}-\frac{8\lambda}{\alpha+2}t^{\frac{\alpha d}{2}}\|u(t)\|_{L^{\alpha+2}}^{\alpha+2}=N(1)-(2-\frac{\alpha d}{2})\int_{1}^{t}\tau^{\frac{\alpha d-6}{2}}\|P_{\tau}u(\tau)\|_{L^{2}}^{2}d\tau.
\end{equation}
Note that $d=1$ and $\alpha=\alpha(d)$ now, applying estimate (\ref{eq24}), we deduce from (\ref{P-conservation1}) that
\begin{eqnarray}
\nonumber &&\frac{8\lambda}{\alpha+2}t^{\frac{\alpha}{2}}\|u(t)\|_{L^{\alpha+2}}^{\alpha+2}+N(1) \\
\nonumber &=&t^{\frac{\alpha-4}{2}}\|P_{t}u\|_{L^{2}}^{2}+(2-\frac{\alpha}{2})\int_{1}^{t}\tau^{\frac{\alpha-6}{2}}\|P_{\tau}u(\tau)\|_{L^{2}}^{2}d\tau \\
\nonumber &\geq& C+C\int_{1}^{t}\tau^{-1}d\tau \geq C+C\log t,
\end{eqnarray}
from which we get immediately
\begin{equation}\label{eq25}
    \|u(t)\|_{L^{\alpha+2}}\geq Ct^{-\frac{\alpha}{2(\alpha+2)}}(\log t)^{\frac{1}{\alpha+2}}
\end{equation}
for some constant $C>0$ and all $t\in[1,\infty)$. Thus by (\ref{rapid decay}), (\ref{eq25}) and (\ref{Lorentz1}) we get immediately
\begin{eqnarray}
 \nonumber \infty&>&\|u\|_{L^{a,\infty}((0,\infty),L^{\alpha+2})}^{a}\geq C\|(1+t)^{-\frac{\alpha}{2(\alpha+2)}}[\log(1+t)]^{\frac{1}{\alpha+2}}\|_{L^{a,\infty}(0,\infty)}^{a} \\
 \nonumber &\geq&C\|(1+t)^{-1}[\log(1+t)]^{\frac{2}{\alpha}}\|_{L^{1,\infty}(0,\infty)},
\end{eqnarray}
which is absurd. This closes our proof for $d=1$.
\end{proof}
The next Theorem indicates a sufficient condition on the initial data $u_{0}\in\Sigma$ at $\alpha=\alpha(d)$, which guarantees the corresponding global solutions have rapid decay as $t\rightarrow\pm\infty$. Thus by Theorem \ref{rapidly decaying}, these initial values $u_{0}$ have scattering states at $\pm\infty$ when $d\geq1$, $d\neq2$.
\begin{thm}\label{Rapid decay data}
Let $\lambda\in\mathbb{R}$, $\alpha=\alpha(d)=\frac{2-d+\sqrt{d^{2}+12d+4}}{2d}$, $d\geq1$, and let
\[a=\frac{2\alpha(\alpha+2)}{4-\alpha(d-2)}.\]
There exists $\varepsilon_{0}>0$ such that, if $u_{0}\in H^{2}\cap \mathcal{F}(H^{2})\subset \Sigma$ and there exists an admissible pair $(\mu,\nu)$ with $\max\{\frac{d(\alpha+2)}{d+\alpha},2\}< \nu<\alpha+2$ satisfying
\begin{equation}\label{eq26}
    \sup_{0\leq t<\infty}(t+1)^{\frac{2}{\mu}}\{\|e^{it\Delta}[(x+2i\nabla)u_{0}]\|_{L^{\nu}}+\|e^{it\Delta}u_{0}\|_{L^{\nu}}\}\leq\varepsilon_{0},
\end{equation}
or respectively,
\begin{equation}\label{eq27}
    \sup_{-\infty< t\leq 0}(|t|+1)^{\frac{2}{\mu}}\{\|e^{it\Delta}[(x+2i\nabla)u_{0}]\|_{L^{\nu}}+\|e^{it\Delta}u_{0}\|_{L^{\nu}}\}\leq\varepsilon_{0},
\end{equation}
then the corresponding maximal solution $u$ of (\ref{eq1}) is a positively(resp. negatively) rapidly decaying solution(see Definition \ref{RDS}). Moreover,
\[u\in L^{q,2}((0,\infty),W^{1,r}(\mathbb{R}^{d})) \,\,\,\,\,\, (resp. \,\,\, u\in L^{q,2}((-\infty,0),W^{1,r}(\mathbb{R}^{d}))),\]
\[Pu\in L^{q,2}((0,\infty),L^{r}(\mathbb{R}^{d})) \,\,\,\,\,\, (resp. \,\,\, Pu\in L^{q,2}((-\infty,0),L^{r}(\mathbb{R}^{d}))),\]
where operator $P=x+2it\nabla$ and $(q,r)$ is an admissible pair with $r=\alpha+2$, therefore scattering state exists in $\Sigma$ at $+\infty$(resp. $-\infty$).
\end{thm}
\begin{proof}
By Remark \ref{rem1}, we need only to deal with the positive time $(0,+\infty)$. Since $\alpha=\alpha(d)$, by energy estimates, it's well known that the solution $u$ is global and bounded in $H^{1}(\mathbb{R}^{d})$, moreover, $u\in C(\mathbb{R},\Sigma)$. Thus we have $v(s,y)$ (the Pseudo-conformal Transformation of $u(t,x)$, see Section 2) defined by (\ref{PC-Trans}) satisfies the nonautonomous integral equation (\ref{nonauto Integral Eq1}) on the interval $(0,1)$, and $v\in C([0,1),\Sigma)$.

Therefore, by applying the dispersive estimates(see \cite{C2}) and H\"{o}lder's inequality to the integral equation (\ref{nonauto Integral Eq1}), we have
\begin{equation}\label{eq29}
    \| v(s)\|_{W^{1,\nu}} \leq \|e^{is\Delta}v_{0}\|_{W^{1,\nu}}+C\int_{0}^{s}(1-\tau)^{\frac{\alpha d-4}{2}}(s-\tau)^{-\frac{2}{\mu}}\|v\|_{L^{\frac{\alpha\nu}{\nu-2}}}^{\alpha}\|v\|_{W^{1,\nu}}d\tau.
\end{equation}
Note that by definition of the Pseudo-conformal Transformation (see (\ref{Variables PC-Trans}), (\ref{PC-Trans}) and (\ref{PC-T1})), one easily verifies that the condition (\ref{eq26}) is equivalent to
\begin{equation}\label{eq28}
    \sup_{s\in[0,1)}\|e^{is\Delta}v_{0}\|_{W^{1,\nu}}\leq \varepsilon_{0}.
\end{equation}
Note also that $\alpha=\alpha(d)$, $\max\{\frac{d(\alpha+2)}{d+\alpha},2\}< \nu<\alpha+2$, so we have $L^{\frac{\alpha\nu}{\nu-2}}(\mathbb{R}^{d})\hookrightarrow W^{1,\nu}(\mathbb{R}^{d})$ and $\frac{4-\alpha d}{2}+\frac{2}{\mu}<1$, hence by H\"{o}lder's inequality we get
\begin{equation}\label{eq30}
    \int_{0}^{s}(1-\tau)^{\frac{\alpha d-4}{2}}(s-\tau)^{-\frac{2}{\mu}}d\tau\leq C(\nu,d).
\end{equation}
Set $\Phi(s)=\sup_{\tau\in[0,s]}\|v(\tau)\|_{W^{1,\nu}}$, for $0<s<1$. Then we can deduce from (\ref{eq29}), (\ref{eq28}) and (\ref{eq30}) immediately that
\begin{equation}\label{continuation argument1}
    \Phi(s)\leq \varepsilon_{0}+C\Phi(s)^{\alpha+1} \,\,\, for \,\,\, all \,\,\, 0<s<1.
\end{equation}
Since $u_{0}\in H^{2}\cap \mathcal{F}(H^{2})$, it follows from a trivial continuation arguments(similar to the proof of Theorem \ref{small data}, we omit the details here) that if $\varepsilon_{0}$ is small enough such that $C(2\varepsilon_{0})^{\alpha+1}< \varepsilon_{0}$, then
\begin{equation}\label{eq31}
    \sup_{s\in[0,1)}\|v(s)\|_{W^{1,\nu}}\leq 2\varepsilon_{0}.
\end{equation}
Since $\max\{\frac{d(\alpha+2)}{d+\alpha},2\}< \nu<\alpha+2$, which implies $L^{\alpha+2}\hookrightarrow W^{1,\nu}$, we have \[\sup_{s\in[0,1)}\|v(s)\|_{L^{\alpha+2}}\leq C\sup_{s\in[0,1)}\|v(s)\|_{W^{1,\nu}}\leq 2C(\nu,d)\varepsilon_{0},\]
combined with identity (\ref{PC-T1}), we infer that
\begin{equation}\label{eq32}
    \|u(t)\|_{L^{\alpha+2}}\leq 2C\varepsilon_{0}(t+1)^{-1/a}.
\end{equation}
Therefore, we have
\begin{equation}\label{eq33}
    \|u\|_{L^{a}((0,\infty),L^{\alpha+2})}\leq 2C(\nu,d)\varepsilon_{0},
\end{equation}
which implies that $u$ is a positively rapidly decaying solution. Thus by Theorem \ref{rapidly decaying}, if $\lambda>0$, $d\neq 2$, scattering state $u^{+}$ exists at $+\infty$.

Furthermore, by using the Strichartz's and H\"{o}lder's estimates in Lorentz spaces(see Lemma \ref{Lorentz3}, Lemma \ref{Lorentz2}), we get there exists K independent of $T$ and $u_{0}$ such that
\begin{equation}\label{eq34}
    \|u\|_{L^{q,2}((0,T),W^{1,r})}\leq K\|u_{0}\|_{H^{1}}+K\|u\|_{L^{a,\infty}((0,T),L^{r})}^{\alpha}\|u\|_{L^{q,2}((0,T),W^{1,r})}
\end{equation}
for every $0<T<\infty$, where $(q,r)$ is an admissible pair with $r=\alpha+2$.
Thus by continuation arguments, one can easily deduce from (\ref{eq33}) and (\ref{eq34}) that if $\varepsilon_{0}$ is sufficiently small such that $2^{\alpha+1}K(C\varepsilon_{0})^{\alpha}<1$, then
\begin{equation}\label{eq35}
    \|u\|_{L^{q,2}((0,\infty),W^{1,r})}\leq 2K\|u_{0}\|_{H^{1}}.
\end{equation}
Similarly, one can easily obtain from Srichartz's estimates and (\ref{eq33}) that
\begin{equation}\label{eq36}
    \|(x+2it\nabla)u\|_{L^{q,2}((0,\infty),L^{r})}\leq 2K\|xu_{0}\|_{L^{2}}.
\end{equation}
Applying Strichartz's estimate in Lorentz spaces(Lemma \ref{Lorentz3}), combined with (\ref{eq33}), (\ref{eq35}) and (\ref{eq36}) we get for any $0<t<s$,
\[\|e^{-it\Delta}u(t)-e^{-is\Delta}u(s)\|_{H^{1}}\leq C\|u\|_{L^{q,2}((t,s),W^{1,r})}\rightarrow0,\]
\[\|x(e^{-it\Delta}u(t)-e^{-is\Delta}u(s))\|_{L^{2}}\leq C\|(x+2it\nabla)u\|_{L^{q,2}((t,s),L^{r})}\rightarrow0,\]
as $t,s\rightarrow \infty$. Therefore, the above two estimates implies that, for general $d\geq1$, the scattering state $u^{+}$ exists in $\Sigma$ at $+\infty$.
\end{proof}
We will apply Theorem \ref{Rapid decay data} to certain type of initial data below. In Corollary \ref{Oscillating data}, we will show that when $\alpha=\alpha(d)$, if the initial data is ``sufficiently oscillating", then the corresponding solution will scatter at $\pm\infty$, note that the $L^{2}$ norm of these initial values is unbounded.
\begin{cor}\label{Oscillating data}
 Assume $\lambda\in \mathbb{R}$, $\alpha=\alpha(d)$, $d\geq1$. For arbitrary $\varphi\in H^{2}\cap \mathcal{F}(H^{2})\subset\Sigma$, given $b\in \mathbb{R}$, and let $\widetilde{u}_{b}$ be the corresponding maximal solutions of (\ref{eq1}) with the initial values $\widetilde{u}_{b,0}=e^{i\Delta}(e^{i\frac{b|x|^{2}}{4}}\varphi)$, then there exists $0<b_{0}<\infty$ such that if $b\geq b_{0}$, the solutions $\widetilde{u}_{b}$ are rapidly decaying as $t\rightarrow\infty$, therefore scattering states $u_{b}^{+}$ exist at $+\infty$ and $\widetilde{u}_{b,0}\in \mathcal{R}_{+}$. The same conclusions also hold for $(-\infty,0)$ provided $b\leq -b_{0}$.
\end{cor}
\begin{proof}
By Remark \ref{rem1}, we need only to deal with the positive time $(0,+\infty)$.
Applying Theorem \ref{Rapid decay data}, the key point is to show that condition (\ref{eq26}) is satisfied for certain special choice of admissible pair $(\mu,\nu)$.

Here we fixed $\mu_{0}=\frac{4(d+2)(\alpha+2)}{\alpha d^{2}}$ and $\nu_{0}=\frac{(d+2)(\alpha+2)}{\alpha+d+2}$ to prove (\ref{eq26}). It's obvious that such choice of $(\mu_{0},\nu_{0})$ satisfies the conditions in Theorem \ref{Rapid decay data}, and $\widetilde{u}_{b,0}=e^{i\Delta}(e^{i\frac{b|x|^{2}}{4}}\varphi)\in H^{2}\cap \mathcal{F}(H^{2})$ for any $b\in\mathbb{R}$ and $\varphi\in H^{2}\cap \mathcal{F}(H^{2})$. A direct calculation, based on the explicit kernel of the Schr\"{o}dinger group(refer to \cite{C2} for a review) shows that
\begin{equation}\label{eq37}
    [e^{it\Delta}(e^{i\frac{b|x|^{2}}{4}}\varphi)](x)=e^{i\frac{b|x|^{2}}{4(1+bt)}}[D_{\frac{1}{1+bt}}e^{i\frac{t}{1+bt}\Delta}\varphi](x),
\end{equation}
where the dilation operator $D_{\beta}$, $\beta>0$, is defined by $D_{\beta}\omega(x)=\beta^{\frac{d}{2}}\omega(\beta x)$. Note that by (\ref{eq37}) and simple calculations, we get
\begin{eqnarray}
  &&\sup_{0\leq t<\infty}(t+1)^{\frac{2}{\mu_{0}}}\{\|e^{it\Delta}[(x+2i\nabla)\widetilde{u}_{b,0}]\|_{L^{\nu_{0}}}+
  \|e^{it\Delta}\widetilde{u}_{b,0}\|_{L^{\nu_{0}}}\} \\
  \nonumber &=&\sup_{1\leq t<\infty}t^{\frac{2}{\mu_{0}}}\{\|e^{it\Delta}[e^{i\frac{b|x|^{2}}{4}}(x\varphi)]\|_{L^{\nu_{0}}}+
  \|e^{it\Delta}(e^{i\frac{b|x|^{2}}{4}}\varphi)\|_{L^{\nu_{0}}}\} \\
  \nonumber &=&\sup_{1\leq t<\infty}\frac{t^{\frac{2}{\mu_{0}}}}{(bt+1)^{\frac{2}{\mu_{0}}}}\{\|e^{i\frac{t}{1+bt}\Delta}(x\varphi)\|_{L^{\nu_{0}}}+
  \|e^{i\frac{t}{1+bt}\Delta}\varphi\|_{L^{\nu_{0}}}\}\\
  \nonumber &\leq&\sup_{1\leq t<\infty}\frac{t^{\frac{2}{\mu_{0}}}}{(bt+1)^{\frac{2}{\mu_{0}}}}C\|\varphi\|_{H^{2}\cap \mathcal{F}(H^{2})}\leq Cb^{-\frac{2}{\mu_{0}}}\longrightarrow 0,
\end{eqnarray}
as $b\rightarrow\infty$. Therefore, there exists a $0<b_{0}<\infty$ such that, if $b\geq b_{0}$, then initial data $\widetilde{u}_{b,0}$ satisfies the condition (\ref{eq26}). The result now follows from Theorem \ref{Rapid decay data}.
\end{proof}
Applying Theorem \ref{rapidly decaying} and Corollary \ref{Oscillating data}, we now characterize the sets $\mathcal{R}^{\pm}$ in the case $\lambda>0$, $\alpha=\alpha(d)$ in terms of rapidly decaying solutions. Our result extends the work of Cazenave and Weissler in \cite{C1} (see Theorem 4.12 therein), which is devoted to the case $\alpha>\alpha(d)$.
\begin{thm}\label{characterization of R}
Assume $\lambda>0$ and $\alpha=\alpha(d)=\frac{2-d+\sqrt{d^{2}+12d+4}}{2d}$, and let $u_{\varphi}$ be the corresponding solution of (\ref{eq1}) with initial data $u(0)=\varphi$. If $d\geq1$, $d\neq2$, $\mathcal{R}_{+}=\{\varphi\in\Sigma; \,\, u_{\varphi} \,\, has \,\, rapid \,\, decay \,\, as \,\, t\rightarrow \infty\}$ and $\mathcal{R}_{-}=\{\varphi\in\Sigma; \,\, u_{\varphi} \,\, has \,\, rapid \,\, decay \,\, as \,\, t\rightarrow -\infty\}$. Moreover, for $d\geq 1$, $\mathcal{R}_{\pm}$ are unbounded subsets of $L^{2}(\mathbb{R}^{d})$.
\end{thm}
\begin{proof}
By Remark \ref{rem1}, we need only show the result for $\mathcal{R}_{+}$. First, if $\varphi\in\mathcal{R}_{+}$, then
\[\|(x+2it\nabla)u_{\varphi}\|_{L^{2}}\leq C,\]
and hence by Gagliardo-Nirenberg's inequality, we have
\[\|u_{\varphi}(t)\|_{L^{\alpha+2}}\leq C(t^{-1}\|(x+2it\nabla)u_{\varphi}\|_{L^{2}})^{-\frac{\alpha d}{2(\alpha+2)}}\leq Ct^{-\frac{\alpha d}{2(\alpha+2)}}.\]
Therefore $\|u_{\varphi}\|_{L^{a,\infty}((0,\infty),L^{\alpha+2})}<\infty$, that is, $u$ has rapid decay as $t\rightarrow \infty$(see Definition \ref{RDS}). Conversely, if $\varphi\in\Sigma$ be such that $u_{\varphi}$ has rapid decay as $t\rightarrow \infty$, then by Theorem \ref{rapidly decaying}, for $d\geq1$, $d\neq2$, the scattering state exists in $\Sigma$ at $+\infty$, i.e., $\varphi\in\mathcal{R}_{+}$. This closes the proof for the first assertion.

Now consider arbitrary $\psi\in H^{2}\cap \mathcal{F}(H^{2})$, and let $\widetilde{u}_{b,0}=e^{i\Delta}(e^{i\frac{b|x|^{2}}{4}}\psi)$, it follows from Corollary \ref{Oscillating data} that $\widetilde{u}_{b,0}\in\mathcal{R}_{+}$, for $b$ large enough. Note that $\|\widetilde{u}_{b,0}\|_{L^{2}}=\|\psi\|_{L^{2}}$, it follows that $\mathcal{R}_{+}$ is unbounded subsets of $L^{2}(\mathbb{R}^{d})$.
\end{proof}
Next, We will investigate the scattering theory in $\Sigma$ for focusing NLS with $\alpha\geq\frac{4}{d}$ and initial data $u_{0}\in\Sigma$ below a mass-energy threshold and satisfying an mass-gradient bound. For the $H^{1}$ scattering results, see \cite{D1,F1,H1}.
\begin{thm}\label{mass supercritical}
Assume $\lambda>0$, $d\geq1$, $\frac{4}{d}\leq\alpha<\frac{4}{d-2}$($\alpha\in [\frac{4}{d}, \infty)$, if $d=1,2$). If initial data $u_{0}\in\Sigma$ satisfies the following assumptions:
\begin{equation}\label{mass}
    \|u_{0}\|_{L^{2}}<\lambda^{-\frac{1}{\alpha}}\|Q\|_{L^{2}}, \,\,\, for \,\,\, \alpha=\frac{4}{d},
\end{equation}
or
\begin{equation}\label{mass super1}
    M[u_{0}]^{\sigma}E[u_{0}]<\lambda^{-2\tau}M[Q]^{\sigma}E[Q],
\end{equation}
\begin{equation}\label{mass super2}
    \|u_{0}\|_{L^{2}}^{\sigma}\|\nabla u_{0}\|_{L^{2}}<\lambda^{-\tau}\|Q\|_{L^{2}}^{\sigma}\|\nabla Q\|_{L^{2}}, \,\,\, for \,\,\, \alpha>\frac{4}{d},
\end{equation}
then scattering states $u^{\pm}$ exist in $\Sigma$ at $\pm\infty$, where $\sigma=\frac{4-(d-2)\alpha}{\alpha d-4}$, $\tau=\frac{2}{\alpha d-4}$ and $Q$ is the ground state solution to $-\Delta Q+Q=|Q|^{\alpha}Q$.
\end{thm}
\begin{proof}
Note that the conditions in Theorem \ref{mass supercritical} is invariant by taking complex conjugation, so by Remark \ref{rem1}, we need only prove the results for positive time $t\rightarrow+\infty$. Throughout our proof, we define \[v(t,x)=e^{-i\frac{|x|^{2}}{4t}}u(t,x),\]
where $u(t,x)$ is the corresponding maximal $H^{1}$ solution of (\ref{eq1}).
We consider separately the cases $\alpha=\frac{4}{d}$ and $\alpha>\frac{4}{d}$. \\
$\mathrm{Case} \,\,\, \alpha=\frac{4}{d}$. It has been proved by M. Weinstein in \cite{W1} that the maximal $H^{1}$ solution of (\ref{eq1}) is global and bounded in $H^{1}(\mathbb{R}^{d})$. To prove the scattering results, we argue by contradiction and assume that $u$ doesn't scatter at $+\infty$. Then by Proposition \ref{blowup1}, we have for all $t\in(0,+\infty)$,
 \begin{equation}\label{eq38}
    \|u(t)\|_{L^{\alpha+2}}\geq C(1+|t|)^{-\frac{2(1-\theta)}{\alpha+2}},
\end{equation}
where $\theta>0$ is defined the same as in Lemma \ref{blowup0}.
Note that the conservation of mass and Gagliardo-Nirenberg's inequality imply that
\begin{equation}\label{eq39}
    \frac{1}{2}\|\nabla v(t)\|_{L^{2}}^{2}\leq E[v(t)]+\frac{\lambda}{\alpha+2}C_{GN}\|\nabla v(t)\|_{L^{2}}^{2}\|u_{0}\|_{L^{2}}^{\alpha},
\end{equation}
and the best constant $C_{GN}=\frac{\alpha+2}{2}\|Q\|_{L^{2}}^{-\alpha}$(see Corollary 2.1 in \cite{W1}), where $Q$ is the ground state solution to $-\Delta Q+Q=|Q|^{\alpha}Q$. Thus we can deduce from (\ref{mass}) and (\ref{eq39}) that
\begin{equation}\label{eq40}
    \|\nabla v(t)\|_{L^{2}}^{2}\leq CE[v(t)].
\end{equation}
Since $\alpha=\frac{4}{d}$, the pseudo-conformal conservation law (\ref{conformal-conservation3}) becomes a precise conservation law
\begin{equation}\label{eq41}
    8t^{2}E[v(t)]=\|xu_{0}\|_{L^{2}}^{2},
\end{equation}
combined with (\ref{eq40}), we get
\[\|u(t)\|_{L^{\alpha+2}}^{\alpha+2}=\|v(t)\|_{L^{\alpha+2}}^{\alpha+2}\leq CE[v(t)]\|u_{0}\|_{L^{2}}^{\alpha}\leq Ct^{-2} \,\,\,\,\, for \,\, all \,\,\, t\in\mathbb{R},\]
which contradicts (\ref{eq38}). This closes our proof for $\alpha=\frac{4}{d}$. \\
$\mathrm{Case} \,\,\, \alpha>\frac{4}{d}$. By scaling, we may assume $\lambda=1$. We first recall some basic properties of $Q$, the ground state of $-\Delta Q+Q=|Q|^{\alpha}Q$, by Pohozaev's identity(see e.g. Corollary 8.1.3 in \cite{C2}),
\begin{equation}\label{ground state}
    \|Q\|_{L^{2}}^{2}=\frac{4-(d-2)\alpha}{\alpha d}\|\nabla Q\|_{L^{2}}^{2}=\frac{4-(d-2)\alpha}{2(\alpha+2)}\|Q\|_{L^{\alpha+2}}^{\alpha+2}.
\end{equation}
Using (\ref{ground state}), the best constant in the Gagliardo-Nirenberg's inequality
\begin{equation}\label{GN}
    \|\omega\|_{L^{\alpha+2}}^{\alpha+2}\leq C_{GN}\|\omega\|_{L^{2}}^{\frac{4-(d-2)\alpha}{2}}\|\nabla \omega\|_{L^{2}}^{\frac{\alpha d}{2}}
\end{equation}
can be given by
\begin{equation}\label{best constant}
    C_{GN}=\frac{2(\alpha+2)}{\alpha d}[\|Q\|_{L^{2}}^{\sigma}\|\nabla Q\|_{L^{2}}]^{-\frac{\alpha d-4}{2}}.
\end{equation}
Define $f(x)=\frac{1}{2}x^{2}-\frac{C_{GN}}{\alpha+2}x^{\frac{\alpha d}{2}}$ for $x\geq 0$. It follows from Gagliardo-Nirenberg's inequality (\ref{GN}), energy conservation and (\ref{mass super1})that
\begin{equation}\label{eq42}
    f(\|u_{0}\|_{L^{2}}^{\sigma}\|\nabla u(t)\|_{L^{2}})\leq M[u_{0}]^{\sigma}E[u_{0}]<M[Q]^{\sigma}E[Q],
\end{equation}
and note that if we take $\omega=Q$, then we get equality in (\ref{GN}), we infer
\begin{equation}\label{eq43}
    f(\|u_{0}\|_{L^{2}}^{\sigma}\|\nabla u(t)\|_{L^{2}})<f(\|Q\|_{L^{2}}^{\sigma}\|\nabla Q\|_{L^{2}}).
\end{equation}
Since $\|Q\|_{L^{2}}^{\sigma}\|\nabla Q\|_{L^{2}}$ is a local maximum point of $f$, we deduce from the continuity of $\|\nabla u(t)\|_{L^{2}}$ in $t$, the initial mass-gradient bound (\ref{mass super2}) and (\ref{eq43}) that
\begin{equation}\label{eq44}
    \|u_{0}\|_{L^{2}}^{\sigma}\|\nabla u(t)\|_{L^{2}}<\|Q\|_{L^{2}}^{\sigma}\|\nabla Q\|_{L^{2}} \,\,\, for \,\,\, all \,\,\, t\in[0,T_{max}),
\end{equation}
which implies the maximal solution $u$ is positively global in time and bounded in $H^{1}(\mathbb{R}^{d})$. Thus we deduce from the Gagliardo-Nirenberg's inequality that
\begin{equation}\label{eq45}
    \|\nabla u(t)\|_{L^{2}}^{2}<\frac{2\alpha d}{\alpha d-4}E[u] \,\,\,\,\,\, for \,\, all \,\,\, t\in[0,+\infty).
\end{equation}
Define a constant $\eta_{0}>0$ by $\eta_{0}=\eta_{0}(u_{0})=\frac{M[u_{0}]^{\sigma}E[u_{0}]}{M[Q]^{\sigma}E[Q]}$, it follows form (\ref{mass super1}) that $\eta_{0}<1$. Using Pohozaev's identity (\ref{ground state}) and (\ref{eq45}), we get
\begin{equation}\label{eq46}
    \|u_{0}\|_{L^{2}}^{\sigma}\|\nabla u(t)\|_{L^{2}}<\sqrt{\frac{2\alpha d}{\alpha d-4}}(M[u]^{\sigma}E[u])^{\frac{1}{2}}<\eta_{0}^{\frac{1}{2}}\|Q\|_{L^{2}}^{\sigma}\|\nabla Q\|_{L^{2}}.
\end{equation}
Note that by the pseudo-conformal conservation law(\ref{conformal-conservation3}), energy conservation and (\ref{eq45}) we have
\begin{eqnarray}\label{important}
   \nonumber \|\nabla v(t)\|_{L^{2}}^{2}&=&\frac{\|xu_{0}\|_{L^{2}}^{2}}{4t^{2}}+\|\nabla u(t)\|_{L^{2}}^{2}-2E[u]+\frac{\alpha d-4}{(\alpha+2)t^{2}}\int_{0}^{t}\tau\|u(\tau)\|_{L^{\alpha+2}}^{\alpha+2}d\tau \\
  &<&\frac{\|xu_{0}\|_{L^{2}}^{2}}{4t^{2}}+\|\nabla u(t)\|_{L^{2}}^{2}-2E[u]+\frac{4}{t^{2}}\int_{0}^{t}\tau E[u]d\tau  \\
  \nonumber &=&\frac{\|xu_{0}\|_{L^{2}}^{2}}{4t^{2}}+\|\nabla u(t)\|_{L^{2}}^{2}.
\end{eqnarray}
Therefore, we deduce easily from (\ref{important}) that,
\begin{equation}\label{eq47}
    \|\nabla v(t)\|_{L^{2}}^{2}<\eta_{0}^{-\frac{1}{2}}\|\nabla u\|_{L^{2}}^{2}
\end{equation}
for all $t\geq t_{0}=\frac{\|xu_{0}\|_{L^{2}}^{2}}{\sqrt{8E[u_{0}]}}[(\frac{M[Q]^{\sigma}E[Q]}
{M[u_{0}]^{\sigma}E[u_{0}]})^{^{1/2}}-1]^{-\frac{1}{2}}$. In view of (\ref{eq46}) and (\ref{eq47}), we get immediately
\begin{equation}\label{eq48}
    \|u_{0}\|_{L^{2}}^{\sigma}\|\nabla v(t)\|_{L^{2}}^{2}<\eta_{0}^{\frac{1}{4}}\|Q\|_{L^{2}}^{\sigma}\|\nabla Q\|_{L^{2}}
\end{equation}
for all $t\geq t_{0}$. Applying Gagliardo-Nirenberg inequality (\ref{GN}) and (\ref{eq48}), we get
\begin{equation}\label{eq49}
    \|u(t)\|_{L^{\alpha+2}}^{\alpha+2}<\frac{2(\alpha+2)}{\alpha d}\eta_{0}^{\frac{\alpha d-4}{8}}\|\nabla v(t)\|_{L^{2}}^{2}<\frac{4(\alpha+2)}{\alpha d-4}\eta_{0}^{\frac{\alpha d-4}{8}}E[v(t)]
\end{equation}
for all $t\geq t_{0}$. Note that $0<\eta_{0}=\frac{M[u_{0}]^{\sigma}E[u_{0}]}{M[Q]^{\sigma}E[Q]}<1$, by the pseudo-conformal conservation law(\ref{conformal-conservation3}), we have
\begin{eqnarray}\label{eq50}
 \nonumber t^{2}\|u(t)\|_{L^{\alpha+2}}^{\alpha+2}&<&\frac{4(\alpha+2)}{\alpha d-4}\eta_{0}^{\frac{\alpha d-4}{8}}t^{2}E[v(t)]\\
 &\leq& \frac{4(\alpha+2)}{\alpha d-4}t_{0}^{2}E[v(t_{0})]+2\eta_{0}^{\frac{\alpha d-4}{8}}\int_{t_{0}}^{t}\tau\|u(\tau)\|_{L^{\alpha+2}}^{\alpha+2}d\tau
\end{eqnarray}
for all $t\geq t_{0}$, therefore we deduce from Gronwall's inequality that
\begin{equation}\label{eq51}
   \|u(t)\|_{L^{\alpha+2}}^{\alpha+2}\leq Ct^{-2(1-\eta_{0}^{\frac{\alpha d-4}{8}})}
\end{equation}
for all $t\geq t_{0}$. Note that $0<\eta_{0}<1$, (\ref{eq51}) implies that
\begin{equation}\label{eq52}
    \|u(t)\|_{L^{\alpha+2}}\rightarrow0,
\end{equation}
as $t\rightarrow+\infty$. Since $\alpha>\frac{4}{d}$, by Strichartz's estimates and continuation arguments, one can easily obtain that
\begin{equation}\label{eq53}
    u\in L^{q}((0,\infty),W^{1,r}(\mathbb{R}^{d})), \,\,\,\,\,\, P_{t}u=(x+2it\nabla)u\in L^{q}((0,\infty),L^{r}(\mathbb{R}^{d}))
\end{equation}
for every admissible pair $(q,r)$, we omit the details here(see e.g. Theorem 7.7.3 in \cite{C2}, and following the proof of this theorem with $P_{t}u$ instead of $u$). In particular, by identity (\ref{PC-T2}), $u\in L^{\infty}((0,\infty),H^{1})$ and $P_{t}u\in L^{\infty}((0,\infty),L^{2})$ implies that
\begin{equation}\label{eq54}
    \sup_{s\in[0,1)}\|v(s)\|_{H^{1}}<\infty,
\end{equation}
where $v(s)$ is the pseudo-conformal transformation of $u(t)$. Thus by Theorem \ref{local existence} and Proposition \ref{equivalent}, scattering state $u^{+}$ exists in $\Sigma$ at $+\infty$.
\end{proof}

\section{Convergence of scattering solution to a free solution}
Finally we study the asymptotic behavior of $\|u(t)-e^{it\Delta}u^{\pm}\|_{\Sigma}$ under the assumption $u^{\pm}$ exist at $\pm\infty$. In general, since $e^{it\Delta}$ is not an isometry of $\Sigma$, it is not known whether we can deduce $\|u(t)-e^{it\Delta}u^{\pm}\|_{\Sigma}\rightarrow 0$ from the scattering asymptotic property $\|e^{-it\Delta}u(t)-u^{\pm}\|_{\Sigma}\rightarrow 0$. A positive answer has been given by B\'{e}gout \cite{B1} for $d\leq 2$, $\alpha>\frac{4}{d}$, and $3\leq d\leq 5$, $\alpha>\frac{8}{d+2}$. Our results in this section extend this work to spatial dimension $d\leq 9$ and wider admissible range of $\alpha$ under certain suitable assumption on initial data $u_{0}$.
\begin{lem}\label{priori estimate}
Assume $\lambda\in\mathbb{R}\setminus \{0\}$, $\frac{2}{d}<\alpha<\frac{4}{d-2}$($\frac{2}{d}<\alpha<\infty$, if $d=1,2$), and $u_{0}\in\mathcal{R}_{\pm}$, $u\in C(\mathbb{R},\Sigma)$ is the corresponding global solution of (\ref{eq1}). Then for any admissible pair $(q,r)$, we have if furthermore $\alpha>\frac{4}{d+2}$, then
\[u\in L^{q}(\mathbb{R},W^{1,r}(\mathbb{R}^{d})).\]
\end{lem}
\begin{proof}
For the proof, refer to B\'{e}gout \cite{B1}, Proposition 3.1.
\end{proof}
\begin{thm}\label{convergence}
Assume $\lambda\in\mathbb{R}\setminus \{0\}$, $d\geq3$, $\frac{2}{d}<\alpha<\frac{4}{d-2}$ and $u_{0}\in\mathcal{R}_{\pm}$, $u\in C(\mathbb{R},\Sigma)$ is the corresponding global solution of (\ref{eq1}) with scattering states $u_{\pm}$ at $\pm\infty$. We assume further that $u_{0}\in\Sigma\cap W^{1,\rho'}$, where $\rho=\frac{4d}{2d-\alpha(d-2)}$, then if $3\leq d\leq9$, $\alpha>\frac{16}{3d+2}$, we have
\[\lim_{t\rightarrow\pm\infty}\|u(t)-e^{it\Delta}u^{\pm}\|_{\Sigma}\rightarrow 0.\]
\end{thm}
\begin{rem}\label{rem3}
Note that for $3\leq d\leq9$, we have $\frac{4}{d}<\frac{16}{3d+2}<\min\{\frac{8}{d+2},\frac{4}{d-2}\}$, however, when $d\geq 10$, we have $\frac{16}{3d+2}\geq \frac{4}{d-2}$.
\end{rem}
\begin{proof}
By Remark \ref{rem1}, we need only prove the results for positive time $t\rightarrow+\infty$, for the negative time $t\rightarrow -\infty$, we can change $u_{0}$ to $\bar{u}_{0}$ and correspondingly change $u(t)$ to $\overline{u(-t)}$ to reach the conclusion. Note also that the Schr\"{o}dinger group is isometric on $H^{1}$, it's obvious that we have $\|u(t)-e^{it\Delta}u^{+}\|_{H^{1}}\rightarrow 0$, as $t\rightarrow\infty$, we need only prove
\[\lim_{t\rightarrow\infty}\|x(u(t)-e^{it\Delta}u^{+})\|_{L^{2}}=0.\]

Let $(\gamma,\rho)$ be an admissible pair with $\rho(\alpha)=\frac{4d}{2d-\alpha(d-2)}$, and let $p=p(\chi)=\chi \gamma
$ and $\tilde{p}$ be such that $\frac{1}{p}+\frac{1}{\tilde{p}}=\frac{2}{\gamma}$, where $\chi$ is any number satisfying $\frac{1}{2}<\chi<1$. Since $u_{0}$ has scattering state $u_{+}$ at $+\infty$, by Proposition \ref{equivalent} and Theorem \ref{local existence}, we have
\begin{equation}\label{eq55}
    \sup_{s\in[0,1)}\|v(s)\|_{H^{1}}<\infty,
\end{equation}
where $v(s)$ is the pseudo-conformal transformation of $u(t)$. Then by Sobolev embedding and (\ref{PC-T1}), we infer
\begin{equation}\label{eq56}
    \sup_{t\in[0,\infty)}(1+t)\|u(t)\|_{L^{\frac{2d}{d-2}}}<\infty.
\end{equation}
By applying the dispersive estimate and Hardy-Littlewood-Sobolev inequality(see \cite{S1}) to the Integral equation (\ref{eq2}), combined with (\ref{eq56}), we get
\begin{eqnarray}
  \nonumber &&\|u\|_{L^{p}((1,\infty),W^{1,\rho})}\\
  \nonumber &\leq&C(\int_{1}^{\infty}t^{-\frac{2p}{\gamma}}dt)^{\frac{1}{p}}\|u_{0}\|_{W^{1,\rho'}}
  +C\|\int_{0}^{t}(t-\tau)^{-\frac{2}{\gamma}}\|u(\tau)\|_{L^{\frac{2d}{d-2}}}^{\alpha}\|u(\tau)\|_{W^{1,\rho}}d\tau\|_{L^{p}(1,\infty)}\\
  \nonumber &\leq&C+C\|\|u\|_{L^{\frac{2d}{d-2}}}^{\alpha}\|u\|_{W^{1,\rho}}\|_{L^{\tilde{p}'}(1,\infty)}\\
  \nonumber &\leq& C+C(\int_{1}^{\infty}(1+\tau)^{-\frac{\alpha p}{p+1-3\chi}}d\tau)^{1+\frac{1}{p}-\frac{3}{\gamma}}\|u\|_{L^{\gamma}((1,\infty),W^{1,\rho})}.
\end{eqnarray}
Since $3\leq d\leq9$, $\alpha>\frac{16}{3d+2}>\frac{8}{d+6}$, one easy easily verifies $\frac{\alpha p}{p+1-3\chi}>1$ for all $\chi\in(\frac{1}{2},1)$, thus we have
\begin{equation}\label{eq57}
    (\int_{1}^{\infty}(1+\tau)^{-\frac{\alpha p}{p+1-3\chi}}d\tau)^{1+\frac{1}{p}-\frac{3}{\gamma}}<\infty.
\end{equation}
Note that $\alpha>\frac{16}{3d+2}>\frac{4}{d+2}$, by Lemma \ref{priori estimate}, we have
\begin{equation}\label{eq58}
    \|u\|_{L^{\gamma}((0,\infty),W^{1,\rho}(\mathbb{R}^{d}))}<\infty.
\end{equation}
Therefore, we deduce from the above three estimates that
\begin{equation}\label{eq59}
    u\in L^{p(\chi)}((0,\infty),W^{1,\rho}(\mathbb{R}^{d}))
\end{equation}
for all $\chi\in(\frac{1}{2},1)$, $\alpha>\frac{16}{3d+2}$.

Now let $H(\chi)=\frac{16\chi}{(d+6)\chi+d-2}$, for $\frac{1}{2}<\chi<1$. One easily verifies that $H(\chi)$ is monotone increasing on the interval $(\frac{1}{2},1)$, and we have
\[\lim_{\chi\rightarrow\frac{1}{2}+}H(\chi)=\frac{16}{3d+2}.\]
Therefore, for arbitrary fixed $\frac{16}{3d+2}<\alpha_{0}<\frac{4}{d-2}$, there exists a $\chi_{0}\in(\frac{1}{2},1)$ such that
\begin{equation}\label{eq60}
    \alpha_{0}>H(\chi_{0}).
\end{equation}
Let the corresponding index $p_{0}=p(\chi_{0})=\chi_{0}\gamma_{0}$ and $\gamma_{0}=\gamma(\alpha_{0})=\frac{8}{(d-2)\alpha_{0}}$, $\rho_{0}=\rho(\alpha_{0})$, from (\ref{eq56}), (\ref{eq59}) and (\ref{eq60}) we deduce that
\begin{eqnarray}\label{eq61}
  &&\||u|^{\alpha_{0}}u\|_{L^{\gamma_{0}'}((t,\infty),W^{1,\rho_{0}'})} \\
  \nonumber &\leq&C(\int_{t}^{\infty}(1+\tau)^{-\frac{\alpha_{0}p_{0}}{p_{0}-1-\chi_{0}}}d\tau)^{1-\frac{1}{\gamma_{0}}-\frac{1}{p_{0}}}
  \|u\|_{L^{p_{0}}((t,\infty),W^{1,\rho_{0}})}\\
  \nonumber &\leq&C(1+t)^{-(\frac{2\alpha_{0}}{H(\chi_{0})}-1)}
\end{eqnarray}
for any $t>0$. Thus by Strichartz's estimates, we have for all $t>0$,
\begin{equation}\label{eq62}
    \|e^{-it\Delta}u(t)-u^{+}\|_{H^{1}}\leq C\||u|^{\alpha_{0}}u\|_{L^{\gamma_{0}'}((t,\infty),W^{1,\rho_{0}'})}\leq C(1+t)^{-(\frac{2\alpha_{0}}{H(\chi_{0})}-1)}.
\end{equation}
Applying the commutative properties of the operator $P_{t}=x+2it\nabla$(see (\ref{P3})), a simple calculation shows that
\begin{equation}\label{eq63}
    x(u(t)-e^{it\Delta}u^{+})=e^{it\Delta}[x(e^{-it\Delta}u(t)-u^{+})+2it\nabla(u^{+}-e^{-it\Delta}u(t))].
\end{equation}
Using (\ref{eq62}) and (\ref{eq63}), we have
\begin{eqnarray}\label{eq64}
   \nonumber \|x(u(t)-e^{it\Delta}u^{+})\|_{L^{2}}&\leq&\|x(e^{-it\Delta}u(t)-u^{+})\|_{L^{2}}+
   2t\|\nabla(e^{-it\Delta}u(t)-u^{+})\|_{L^{2}} \\
  &\leq&\|x(e^{-it\Delta}u(t)-u^{+})\|_{L^{2}}+C(1+t)^{-2(\frac{\alpha_{0}}{H(\chi_{0})}-1)}.
\end{eqnarray}
Note that $u_{0}\in\mathcal{R}_{+}$, so we get
\begin{equation}\label{eq65}
    \|x(e^{-it\Delta}u(t)-u^{+})\|_{L^{2}}\rightarrow 0,
\end{equation}
as $t\rightarrow\infty$, note also by (\ref{eq60}), we have $\frac{\alpha_{0}}{H(\chi_{0})}-1>0$, thus we deduce from (\ref{eq64}) that
\begin{equation}\label{eq66}
    \|x(u(t)-e^{it\Delta}u^{+})\|_{L^{2}}\rightarrow 0,
\end{equation}
as $t\rightarrow \infty$, it's the convergence that we need. Note that $\alpha_{0}$ is an arbitrary number satisfying $\alpha_{0}\in(\frac{16}{3d+2},\frac{4}{d-2})$, thus we have proved the conclusion for all $\frac{16}{3d+2}<\alpha<\frac{4}{d-2}$, $3\leq d\leq9$.
\end{proof}

{\bf Acknowledgements:} The author's research is supported by a Young Researcher's Fellowship of the Academy of Mathematics and Systems Science, Chinese Academy of Sciences.\\

\end{document}